\documentclass[10pt,]{amsart}

\usepackage[utf8]{inputenc}
\usepackage[T1]{fontenc}
\usepackage[francais]{babel}

\usepackage{graphicx}
\usepackage{amsmath}
\usepackage{amsfonts}
\usepackage{amssymb}
\usepackage{amsthm}
\usepackage{vmargin}

\usepackage[colorlinks=true]{hyperref} 

 \usepackage[all]{xy}
\usepackage{color}

\newtheorem{theorem}{Th\'eor\`eme}[section]

\newtheorem{corollary}[theorem]{Corollaire}
\newtheorem{lemma}[theorem]{Lemme}
\newtheorem{proposition}[theorem]{Proposition}
\newtheorem{definition}[theorem]{D\'efinition}

\newtheorem*{remark}{Remarque}
\newtheorem{example}[theorem]{Exemple}

\newtheorem{thmB}{Th\'eor\`eme}

\newtheorem{corB}{Corollaire}

\renewcommand{\thethmB}{}

\renewcommand{\thecorB}{}

\newcommand{\e}{\epsilon}

\renewcommand{\textbf}[1]{\begingroup\bfseries\mathversion{bold}#1\endgroup}
\def\tdiff{\mathop{\rm tdiff}\nolimits}
\def\trdeg{\mathop{\rm trdeg}\nolimits}

\title{Sp\'ecialisation du groupo\"ide de Galois d'un champ de vecteurs}
\date{\today}

\author{Guy Casale}
\address{Guy Casale, Univ Rennes, CNRS, IRMAR-UMR 6625, F-35000 Rennes, France 
}
\email{\tt guy.casale@univ-rennes1.fr}

\author{Damien Davy}
\address{Damien Davy, Univ Rennes, CNRS, IRMAR-UMR 6625, F-35000 Rennes, France
}
\email{\tt brandavy@outlook.fr}



\begin{document}
\maketitle

\begin{abstract}
Nous montrons un r\'esultat de ``semi-continuit\'e'' du groupo\"ide de Galois d'un champs de vecteur d\'ependant d'un param\`etre. Appliqu\'e aux \'equations de Painlev\'e, ce r\'esultat nous permet de calculer le groupo\"ide de Galois de ces \'equations pour des valeurs g\'en\'erales des param\`etres.  
\end{abstract}

\section{Introduction}

\subsection{Th\'eories de Galois diff\'erentielles et \'equations de Painlev\'e}

\`A la fin du dix-neuvi\`eme si\`ecle, les id\'ees d'\'E. Galois ont \'et\'e \'etendues aux \'equations diff\'erentielles lin\'eaires par \'E. Picard 
\cite{Picard} puis ont \'et\'e compl\'et\'ees par E. Vessiot \cite{Vessiot-memoire}. Dans les ann\'ees 1950,
E. Kolchin d\'eveloppe cette th\'eorie du point de vue des extensions de corps diff\'erentiels \cite{BB,Kolchin2}.    

Dans \cite{Drach-these}, J. Drach avance une th\'eorie de Galois pour les \'equations diff\'erentielles non-lin\'eaires. Malgr\'e les erreurs qui invalident la 
plupart de ses d\'efinitions, il donne des indications pour calculer le groupe de rationalit\'e des \'equations de Painlev\'e \cite{Drach-painleve}.
Dans \cite{Vessiot-definition}, E. Vessiot esquisse une d\'efinition rigoureuse.  Elle est \`a l'origine de la 
d\'efinition du groupe de Galois infinit\'esimal de H. Umemura \cite{Umemura-def} g\'en\'eralisant le groupe de Picard-Vessiot d'un syst\`eme diff\'erentiel lin\'eaire. Le groupo\"ide de Galois d'un feuilletage a \'et\'e introduit par B. Malgrange dans \cite{Malgrange}. Cette d\'efinition
concerne les feuilletages holomorphes (singuliers) sur une vari\'et\'e $\mathbb{C}$-analytique lisse. 
Cet objet g\'en\'eralise
 le groupe de Galois intrins\`eque d'un syst\`eme diff\'erentiel lin\'eaire \cite{Bertrand}.

Les premi\`eres tentatives de calcul du groupo\"ide de Galois de la premi\`ere \'equation de Painlev\'e sont d\^us \`a P. Painlev\'e 
\cite{Painleve-note} et J. Drach \cite{Drach-painleve}. Ces calculs ont \'et\'e rendus rigoureux dans \cite{CasaleP1} en utilisant la classification des pseudo-groupes de Lie agissant sur $\mathbb C ^2$ d\^ue \`a E. Cartan \cite{Cartan}.\\

Dans les Le\c{c}ons de Stockholm \cite{leconsdestockholm}, P. Painlev\'e d\'efinit 
une notion de r\'eductibilit\'e d'une solution d'une \'equation diff\'erentielle. Une \'equation est dite r\'eductible si sa solution g\'en\'erale l'est.

Cette d\'efinition est tr\`es restrictive comme le montre P. Painlev\'e dans la remarque 28 de \cite{Painleve-bulletin}. Elle a n\'eanmoins 
l'int\'er\^et de faire appara\^itre les diff\'erences entre la r\'eductibilit\'e d'une \'equation (ou du feuilletage sous-jacent) et celle d'une solution 
particuli\`ere.
L'\'etude de la r\'eductibilit\'e des solutions particuli\`eres des \'equations de Painlev\'e est l'{\oe}uvre de l'\'ecole japonaise. H. Umemura 
\cite{Umemura1} et K. Nishioka \cite{Nishioka} donnent un crit\`ere permettant de trouver les familles \`a un param\`etre de solutions r\'eductibles d'une 
\'equation du second ordre et l'appliquent \`a l'\'etude de la premi\`ere \'equation de Painlev\'e. \`A la suite de ces articles, Murata \cite{Murata}, 
Watanabe \cite{Watanabe5, Watanabe6}, Noumi-Okamoto \cite{Noumi} et Umemura-Watanabe \cite{Umemura24, Umemura3} trouvent les solutions 
r\'eductibles non alg\'ebriques des autres \'equations de Painlev\'e. 

Ces r\'esultats prouvent l'irr\'eductibilit\'e (au sens des Le\c{c}ons de Stockholm) des solutions des \'equations de Painlev\'e pour des valeurs g\'en\'eriques des param\`etres. Dans \cite{Painleve-bulletin} P. Painlev\'e pose la question du rapport entre une d\'efinition de l'irr\'eductibilit\'e d'une \'equation differentielle et le groupe de rationnalit\'e de J. Drach. Dans \cite{Casaleirred}, le premier auteur r\'epond en partie \`a cette question. Consid\'erons un champ de vecteurs rationnel $X$ sur une vari\'et\'e de dimension trois mod\'elisant une \'equation diff\'erentielle du second ordre. Si le corps des invariants diff\'erentiels rationnels de $X$ est engendr\'e par les invariants provenant d'une forme ``temps'' $dx$, d'une forme ``volume'' $dvol$ invariante et de $X$ lui m\^eme alors l'\'equation diff\'erentielle est irr\'eductible.\\       
Ces conditions ont \'et\'e verifi\'ees pour $P_{{\textsc{\romannumeral 1}}}$ dans \cite{CasaleP1}, $P_{{\textsc{\romannumeral 2}}}(0)$ dans  \cite{CasaleWeil} et $P_{{\textsc{\romannumeral 6}}}(\alpha, \beta,\gamma,\delta)$ pour $(\alpha, \beta, \gamma, \delta) \not = (0,0,0,\frac{1}{2})$ dans \cite{Cantat}. Le cas $(\alpha, \beta, \gamma, \delta) = (0,0,0,\frac{1}{2})$ est trait\'e dans \cite{Casale6}.

Les \'equations de Painlev\'e apparaissent naturellement en famille. Nous utiliserons \cite{Ohyama} comme r\'ef\'erence.
$$
\begin{array}{lrcl}
 P_{{\textsc{\romannumeral 1}}}&   u''&=&6u^2+x\\
P_{{\textsc{\romannumeral 2}}}(\alpha) & u''&=&2u^3+xu+\alpha \\
 P_{{\textsc{\romannumeral 3}}}(\alpha,\beta,\gamma,\delta)& u''&=&\frac{u'^2}{u}-\frac{u'}{x}+\frac{\alpha u^2+\beta}{x}+\gamma u^3+\frac{\delta}{u}\\
 P_{{\textsc{\romannumeral 4}}}(\alpha,\beta)  & u''&=&\frac{u'^2}{2u}+\frac{3}{2}u^3+4xu^2+2(t^2-\alpha)u+\frac{\beta}{u}\\
 P_{{\textsc{\romannumeral 5}}}(\alpha,\beta,\gamma,\delta)& u''&=&\left(\frac{1}{2u}+\frac{1}{u-1}\right)u'^2-\frac{u'}{x}+\frac{(u-1)^2}{x^2}\left(\alpha u+\frac{\beta}{u}\right)+\gamma \frac{u}{x}+\delta\frac{u(u+1)}{u-1}\\
 P_{{\textsc{\romannumeral 6}}}(\alpha,\beta,\gamma,\delta) & u''&=&\frac{1}{2}\left(\frac{1}{u}+\frac{1}{u-1}+\frac{1}{u-x}\right)u'^2-\left(\frac{1}{x}+\frac{1}{x-1}+\frac{1}{u-x}\right)u' \\
&&& \hfill +\frac{u(u-1)(u-x)}{x^2(x-1)^2}\left(\alpha+\beta\frac{x}{u^2}+\gamma\frac{x-1}{(u-1)^2}+\delta\frac{x(x-1)}{(u-x)^2}\right)
\end{array}
$$

Nous nous proposons d'\'etudier les variations du groupo\"ide de Galois d'une famille d'\'equations diff\'erentielles en fonction des param\`etres et d'en d\'eduire les r\'esultats d'irreductibilit\'e des \'equations de Painlev\'e. 

Dans le cas d'\'equations diff\'erentielles lin\'eaires, ces variations ont \'et\'e \'etudi\'ees par L. Goldmann \cite{Goldman} puis M.F. Singer \cite{Singer}. En particulier, ces articles montrent que la dimension du groupe de Picard-Vessiot d'un syst\`eme lin\'eaire diff\'erentiel d\'ependant de param\`etres varie semi-continument  inf\'erieurement avec les param\`etres. Ces r\'esultats ont \'et\'e \'etendus \`a des situations tr\`es g\'en\'erales par Y. Andr\'e \cite{Andre} incluant les confluences d'equations aux diff\'erences vers des \'equations diff\'erentielles.

\subsection{Les r\'esultats}

Consid\'erons la deuxi\`eme \'equation de Painlev\'e:
\[
\left \{\begin{array}{ccl}\tag{$P_{{\textsc{\romannumeral 2}}}$}
u''&=&2u^3+xu+\alpha\\
x'&=&1\\
\alpha'&=&0 
\end{array} \right.
\]
Cette \'equation peut \^{e}tre vue comme un champ de vecteurs sur $\mathbb{C}^4$
\[
X_{{\textsc{\romannumeral 2}}}=\frac{\partial}{\partial x}+v\frac{\partial}{\partial u}+(2u^3+xu+\alpha)\frac{\partial}{\partial v}
\]
Pour une valeur de $\alpha$ fix\'ee \`a $\alpha_0$, l'\'equation
\[
u''=2u^3+xu+\alpha_0 \tag{$P_{{\textsc{\romannumeral 2}}}(\alpha_0)$}
\]
peut \^{e}tre vue comme le champ de vecteurs sur $\mathbb{C}^3$
\[
X_{{\textsc{\romannumeral 2}}}\vert_{\alpha_0}=\frac{\partial}{\partial x}+v\frac{\partial}{\partial u}+(2u^3+xu+\alpha_0)\frac{\partial}{\partial v}
\]
Notre but est de comparer les groupo\"ides de Galois des champs $X_{{\textsc{\romannumeral 2}}}\vert_{\alpha_0}$ pour diff\'erentes valeurs de $\alpha_0$. Pour ce faire, nous les comparerons au groupo\"ide de Galois de $X_{{\textsc{\romannumeral 2}}}$.

La taille du groupo\"ide de Galois d'un champ de vecteurs sera mesur\'e par son type diff\'erentiel, not\'e $\tdiff$, d\'efini par E. R. Kolchin dans (\cite{Kolchin2}). 
Nous obtenons le r\'esultat suivant :

\renewcommand{\thethmB}{\ref{thmcroissance}}
\begin{thmB}
Soient $\rho:M\rightarrow S$ un morphisme lisse \`a fibres connexes entre deux vari\'et\'es lisses irr\'eductibles, $q_0\in S$ et $X$ un champ de vecteurs rationnel tangent aux fibres de $\rho$. Pour $q\in S$ g\'en\'eral,
\[
\tdiff(Gal(X\vert_{q_0}))\leq \tdiff(Gal(X\vert_q))
\]
\end{thmB}

Une propri\'et\'e de $q$ est dite g\'en\'erale dans l'espace des param\`etres $S$ si elle est vraie en dehors d'une union d\'enombrable de sous-vari\'et\'es ferm\'ees strictes de $S$.

Le type diff\'erentiel du groupo\"ide de Galois d'une \'equation d'ordre $2$ est toujours inf\'erieur ou \'egal \`a $2$. Nous savons par ailleurs \cite{CasaleWeil} que $\tdiff(Gal(P_{{\textsc{\romannumeral 2}}}(0)))=2$. Le th\'eor\`eme pr\'ec\'edent montre ainsi que pour $\alpha\in\mathbb{C}$ g\'en\'eral,  $\tdiff(Gal(P_{{\textsc{\romannumeral 2}}}(\alpha)))=2$.
La classification d'E. Cartan des pseudo-groupes de Lie de transformations du plan dans (\cite{Cartan}) donne le r\'esultat suivant:

\renewcommand{\thethmB}{\ref{corvolume}}
\begin{thmB}
Consid\'erons l'\'equation du second ordre
\begin{equation}\tag{$E$}
\frac{d^2u}{dx^2}=F\left(x,u,\frac{du}{dx}\right)
\end{equation}
o\`{u} $F\in \mathbb{C}(x,u,v)$. Soient $X_F$ le champ de vecteurs associ\'e, $dvol$ une $3$-forme rationnelle sur $\mathbb{C}^3$. Notons 
\[
Vol(E):=\{\widehat{\phi} \in \widehat{Aut}(\mathbb C^3)\ \vert\ \widehat{\phi}^*dx=dx,\ \widehat{\phi}^*X_F=X_F \text{ et } \widehat{\phi}^*dvol=dvol\}
\]
Si $\mathcal{L}_{X_F}dvol=0$ et $\tdiff(Gal(E))=2$, alors $Gal(E)=Vol(E)$.
\end{thmB}

Les champs $X_{{\textsc{\romannumeral 2}}}\vert_{\alpha}$ v\'erifiant les hypoth\`eses du th\'eor\`eme,  nous obtenons :
 
 \renewcommand{\thecorB}{\ref{2emepainleve}}
\begin{corB}
Pour $\alpha\in\mathbb{C}$ g\'en\'eral, $Gal(P_{{\textsc{\romannumeral 2}}}(\alpha))=Vol(P_{{\textsc{\romannumeral 2}}}(\alpha))$.
\end{corB}
Associ\'e au r\'esultat de \cite{Casaleirred} donnant un lien entre la structure du groupo\"ide de Galois et irr\'eductibilit\'e d'une \'equation, nous en d\'eduisons: ,
\renewcommand{\thecorB}{\ref{2emepainleve}}
\begin{corB}
Pour $\alpha\in\mathbb{C}$ g\'en\'eral, $P_{{\textsc{\romannumeral 2}}}(\alpha)$ est irr\'eductible.
\end{corB}

Ces r\'esultats obtenus sur la deuxi\`eme \'equation de Painlev\'e s'\'etendent aux autres \'equations de Painlev\'e. Ces \'equations d\'eg\'en\`erent les unes sur les autres en suivant le diagramme
(voir \cite{Ohyama}):
$$
\xymatrix{
\mathbf{P_{{\textsc{\romannumeral 6}}}}\ar[r] &\mathbf{P_{{\textsc{\romannumeral 5}}}}\ar[r]\ar[d]&\mathbf{P_{{\textsc{\romannumeral 3}}}}\ar[d]& \\
&\mathbf{P_{{\textsc{\romannumeral 4}}}}\ar[r]& \mathbf{P_{{\textsc{\romannumeral 2}}}}\ar[r]&\mathbf{P_{{\textsc{\romannumeral 1}}}} \\
}
$$
Nous utilisons cette d\'eg\'en\'erescence pour montrer: 
\renewcommand{\thethmB}{\ref{corirred}}
\begin{thmB}
Pour $J={\textsc{\romannumeral 1}},{\textsc{\romannumeral 2}},{\textsc{\romannumeral 3}},{\textsc{\romannumeral 4}},{\textsc{\romannumeral 5}},{\textsc{\romannumeral 6}}$ et pour $q$ g\'en\'eral dans l'espace des param\`etres de $P_J$, $Gal(P_J(q))=Vol(P_J(q))$.
\end{thmB}
Gr\^{a}ce \`a \cite{Casaleirred}, nous retrouvons un r\'esultat connu (voir \cite{Lisovyy, Murata, Nishioka, Noumi, Umemura1, Umemura24, Umemura3, Watanabe5, Watanabe6}): 
\renewcommand{\thecorB}{\ref{coroirred}}
\begin{corB}
Pour $J={\textsc{\romannumeral 1}},{\textsc{\romannumeral 2}},{\textsc{\romannumeral 3}},{\textsc{\romannumeral 4}},{\textsc{\romannumeral 5}},{\textsc{\romannumeral 6}}$ et pour $q$ g\'en\'eral dans l'espace des param\`etres de $P_J$, l'\'equation $P_J(q)$ est irr\'eductible au sens de Nishioka-Umemura.
\end{corB}

\subsection{Organisation de l'article}

Cet article est compos\'e de quatre parties. Dans la premi\`ere partie, nous rappelons les d\'efinitions des objets avec lesquels nous travaillons. Nous d\'efinissons des espaces de jets d'applications \`a valeurs dans les fibres d'un morphisme $M\rightarrow S$, l'ouvert des rep\`eres des fibres $R(M/S)$ et le groupo\"{i}de associ\'e $Aut(M/S)$. Nous donnons leurs structures alg\'ebriques et d\'ecrivons les prolongements canoniques d'un champ de vecteurs tangent aux fibres de $M/S$ \`a ces espaces de jets.

Dans une deuxi\`eme partie, nous d\'efinissons trois sous-espaces de $Aut(M/S)$
\begin{enumerate}
\item  la plus petite sous-vari\'et\'e tangente au prolongement du champ de vecteurs sur $Aut(M/S)$ et contenant l'identit\'e
\item  le plus petit sous-groupo\"{i}de tangent au prolongement du champ de vecteurs sur $Aut(M/S)$
\item  la sous-vari\'et\'e constitu\'ee des automorphismes pr\'eservant les int\'egrales premi\`eres rationnelles du prolongement du champ de vecteurs sur $R(M/S)$
\end{enumerate}
 Nous montrons que ces trois d\'efinitions co\"{i}ncident en utilisant les r\'esultats de P. Bonnet \cite{Bonnet} et nommons cet objet le groupo\"ide de Galois de $X$ sur $S$.

  Lorsqu'il n'y a pas de param\`etres, {\it i.e.} $S$ est un point, nous montrons que ce groupo\"ide est le groupo\"ide de Galois de B. Malgrange d\'efini dans \cite{Malgrange}. 

Dans la troisi\`eme partie, nous comparons les "tailles" de diff\'erents groupo\"ide de Galois en terme de type diff\'erentiel. Nous montrons le th\'eor\`eme \ref{thmsp\'ec} de sp\'ecialisation dont la cons\'equence imm\'ediate est la "semi-continuit\'e" du type diff\'erentiel du groupo\"ide de Galois par rapport aux param\`etres : th\'eor\`eme \ref{thmcroissance}.

Dans la derni\`ere partie, nous donnons des applications du r\'esultat de sp\'ecialisation aux \'equations diff\'erentielles du second ordre d\'ependant de param\`etres, notamment aux \'equations de Painlev\'e. 

\section{Les espaces de rep\`eres}
Soient $S$ et $M$ deux vari\'et\'es alg\'ebriques complexes, affines, lisses, connexes, de dimensions respectives $d$, $m + d$. Soit $\rho:
M \rightarrow S$ un morphisme lisse \`a fibres connexes. Nous noterons $M_q$ la fibre du morphisme en $q\in S$.
 La vari\'et\'e $S$ sera appel\'ee espace des param\`etres.

\subsection{Le fibr\'e principal des $S$-rep\`eres}


 \begin{definition}
Un $S$-rep\`ere en $p\in M$ est une application formelle $r:(\mathbb{C}^m,0)\rightarrow(M,p)$ telle que:
\begin{itemize}
\item la matrice jacobienne $D_0 r$ soit de rang $m$
\item $\rho\circ r$ soit l'application constante \'egale \`a $\rho(p)$ 
\end{itemize} 
Un $S$-rep\`ere d'ordre $k$ est le jet d'ordre $k$ d'un $S$-rep\`ere.
\noindent Si $S$ est un point, nous dirons que $r$ est un rep\`ere sur $M$.
\end{definition}

\subsubsection{D\'efinition}

Soit $U$ un ouvert de Zarisky  contenant $p\in M$ qui soit un rev\^{e}tement non ramifi\'e d'un ouvert de $\mathbb{C}^{m+d}$ et qui fait commuter le diagramme :
\[
\xymatrix{
    U\ar[r] \ar[d] & \mathbb{C}^{m+d}\ar[d] \\
    S\ar[r]&\mathbb{C}^d}
\] 
Nous appellerons $U$ une carte de $M/S$ en $p$. 
 Il vient un syst\`eme de coordonn\'ees $p=(p',q)\in \mathbb{C}^m\times \mathbb{C}^d$ dans lequel un $S$-rep\`ere s'\'ecrit: 
 \begin{equation}\label{serie}
 r(\e_1,\ldots,\e_m)=\left(\underset{\alpha\in \mathbb{N}^m}{\sum}r_1^{\alpha}\frac{\epsilon^{\alpha}}{\alpha !},\ldots,\underset{\alpha\in \mathbb{N}^m}{\sum}r_m^{\alpha}\frac{\epsilon^{\alpha}}{\alpha !},q\right)
 \end{equation}
  o\`{u}, pour $1\leq i\leq m$, $r_i^{\alpha}\in \mathbb{C}$, $\epsilon^{\alpha}=\epsilon_1^{\alpha_1}\ldots\epsilon_m^{\alpha_m} $ et $\alpha !=\alpha_1 !\ldots\alpha_m !$. Notons $1_j$ le multi-indice ayant $1$ sur la $j$-\`eme coordonn\'ee et $0$ sur les autres. Dans ces coordonn\'ees, $D_0 r$ est la matrice $\left(r_i^{1_j}\right)_{ij}$ compl\'et\'ee par z\'ero sur les $d$ derni\`eres lignes.  Nous noterons $jac(r)=\det(r_i^{1_j})$. Le $m$-uplet de s\'eries (\ref{serie}) tronqu\'ees \`a l'ordre $k\in\mathbb{N}$ est un $S$-rep\`ere d'ordre $k$.

\begin{definition}
L'espace des $S$-rep\`eres est $R(M/S)$ et l'espace des $S$-rep\`eres d'ordre $k\in\mathbb{N}$ est $R_k(M/S)$.  
\end{definition}

Par d\'efinition, l'espace des $S$-rep\`eres est la limite projective des espaces des $S$-rep\`eres d'ordre $k\in\mathbb{N}$.

\subsubsection{Structure de vari\'et\'e pro-alg\'ebrique}\label{ssecalgrep}

La description locale de la structure alg\'ebrique de $R(M/S)$ se fait de la mani\`ere suivante.

Prenons $U\subset M$ une carte de $M/S$ en $p$, de coordonn\'ees $y_1,\ldots,y_m,z_1,\ldots,z_d$. L'ensemble des $S$-rep\`eres d'ordre $k\in\mathbb{N}$ sur $U$ est muni d'une structure de vari\'et\'e affine dont l'anneau de coordonn\'ees est 
\[
\mathcal{O}_M(U)[y_i^{\alpha},\ 1\leq i\leq m,\ \alpha\in \mathbb{N}^m,\ 1\leq\vert\alpha\vert\leq k][1/\det(y_i^{1_j})],
\] 
o\`{u} $y_i^\alpha$ est la fonction d\'efinie par $y_i^\alpha (r) = r_i^\alpha$. L'anneau de coordonn\'ees de la limite projective $R(U/S)$ est 
\[
\mathcal{O}_M(U)[y_i^{\alpha},\ 1\leq i\leq m,\ \alpha\in \mathbb{N}^m][1/\det(y_i^{1_j})],
\] 
Les d\'erivations $\partial_1,\ldots,\partial_m$ d\'efinies par $\partial_iy_j^\alpha:=y_j^{\alpha+1_j},\ \partial_iz_j=0$ donne une structure de $\mathbb{C}[\partial_1,\ldots,\partial_m]$-alg\`ebre diff\'erentielle \`a cet anneau.

Cette structure admet une description globale. Le faisceau structural de $RM$ est enti\`erement d\'etermin\'ee par son image directe sur $M$ par $R(M/S)\rightarrow M$. Par abus, nous parlerons indiff\'eremment du faisceau structural de $R(M/S)$ et de son image directe sur $M$.

Notons $Jac$ le faisceau coh\'erent d'id\'eaux engendr\'e par $\det(y_i^{1_j})$ o\`{u} $y_1,\ldots,y_m$ sont les coordonn\'es sur une carte $U$ de $M/S$. Ce faisceau en id\'eaux est bien d\'efini car si $\tilde{y}_1,\ldots,\tilde{y}_m$ sont d'autres coordonn\'ees, alors $\det(\tilde{y}_i^{1_j})=c \det(y_i^{1_j})$ o\`{u} $c$ est une fonction non nulle.

\begin{definition}
Le faisceau $\mathcal{O}_{R(M/S)}$ des coordonn\'ees de l'espace des $S$-rep\`eres est:
\[
\left(Sym(\mathbb{C}[\partial_1,\ldots,\partial_m]\otimes \mathcal{O}_M)/\mathcal{L}\right)[1/Jac]
\]
o\`{u} 
\begin{itemize}
\item $\mathbb{C}[\partial_1,\ldots,\partial_m]\otimes \mathcal{O}_M$ est un produit tensoriel de $\mathbb{C}$-espaces vectoriels,
\item $Sym(\mathbb{C}[\partial_1,\ldots,\partial_m]\otimes \mathcal{O}_M)$ est l'alg\`ebre sym\'etrique engendr\'ee l'espace vectoriel pr\'ec\'edent,
\item cette alg\`ebre a une structure de $\mathbb{C}[\partial_1,\ldots,\partial_m]$-alg\`ebre diff\'erentielle \`a gauche,
\item elle est quotient\'ee par l'id\'eal diff\'erentiel $\mathcal{L}$ engendr\'e par le noyau de $Sym(1\otimes_{\mathbb{C}}\mathcal{O}_M)\rightarrow\mathcal{O}_M$ et par les $\partial_i\otimes\mathcal{O}_S$,
\item elle est enfin localis\'ee par $Jac$
\end{itemize}
Le faisceau $\mathcal{O}_{R_k(M/S)}$ des coordonn\'ees de l'espace des $S$-rep\`eres d'ordre $k\in\mathbb{N}$ est:
\[
(Sym(\mathbb{C}[\partial_1,\ldots,\partial_m]^{\leq k}\otimes \mathcal{O}_M)/(\mathcal{L}\cap Sym(\mathbb{C}[\partial_1,\ldots,\partial_m]^{\leq k}\otimes \mathcal{O}_M))[1/jac]
\]
o\`{u} $\mathbb{C}[\partial_1,\ldots,\partial_m]^{\leq k}$ d\'esigne l'ensemble des polyn\^{o}mes en les $\partial_i$ de degr\'e au plus $k$.
\end{definition}

Les espaces des $S$-rep\`eres d'ordre $k\in\mathbb{N}$ sont des vari\'et\'es irr\'eductibles et l'espace des $S$-rep\`eres est une vari\'et\'e pro-alg\'ebrique. 

Nous retrouvons la structure locale, sur une carte de $M/S$. Soit $U\subset M$ une carte de $M/S$ de coordonn\'ees $y_1,\ldots,y_m,z_1,\ldots,z_d$ :

\begin{lemma} \label{lemreperelocglob}
 Nous avons un isomorphisme 
\[
\mathcal{O}_{R(M/S)}(U)\simeq\mathcal{O}_M(U)[y_i^{\alpha}\mbox{, }1\leq i\leq m\mbox{, }\alpha\in \mathbb{N}^m\mbox{, }1\leq\vert\alpha][1/det(y_i^{1_j})]
\] 
qui est compatible aux actions de $\mathbb{C}[\partial_1,\ldots,\partial_m]$. La filtration de $\mathcal{O}_{R(M/S)}(U)$ par l'ordre des op\'erateurs diff\'erentiel co\"incide avec la filtration par le poids $|\alpha|$ des nouvelles variables $y_i^\alpha$.

\end{lemma}

\begin{proof}
Commen\c{c}ons par d\'efinir un morphisme d'espace vectoriel \`a partir de l'injection canonique:
\[
\begin{array}{cccc}
\phi :&1\otimes \mathcal{O}_M(U)&\longrightarrow &\mathcal{O}_M(U)[y_i^{\alpha}\mbox{, } 1\leq i\leq m\mbox{, }\alpha\in \mathbb{N}^m\mbox{, }1\leq\vert\alpha\vert]\\
&1\otimes P&\longmapsto &P
\end{array}
\]
Nous allons le prolonger en utilisant la structure de $\mathbb{C}[\partial_1,\ldots,\partial_m]$- module puis en utilisant la structure d'anneau de l'espace d'arriv\'e. 
L'anneau $\mathbb{C}[\partial_1,\ldots,\partial_m]$ agit sur $1\otimes \mathcal{O}_M(U)$ par multiplication sur le premier facteur. Cette application se prolonge en un morphisme de $\mathbb{C}[\partial_1,\ldots,\partial_m]$-module 
\[
\mathbb{C}[\partial_1,\ldots,\partial_m]\otimes \mathcal{O}_M(U)\longrightarrow \mathcal{O}_M(U)[y_i^{\alpha}\mbox{, } 1\leq i\leq m\mbox{, }\alpha\in \mathbb{N}^m\mbox{, }1\leq\vert\alpha\vert]
\]
Nous en d\'eduisons un morphisme d'alg\`ebre
\[
Sym(\mathbb{C}[\partial_1,\ldots,\partial_m]\otimes \mathcal{O}_M(U))\longrightarrow \mathcal{O}_M(U)[y_i^{\alpha}\mbox{, } 1\leq i\leq m\mbox{, }\alpha\in \mathbb{N}^m\mbox{, }1\leq\vert\alpha\vert]
\]
%
 Il est encore compatible aux actions des op\'erateurs diff\'erentiels de $M$ car elles doivent toutes deux v\'erifier la r\`egle de Leibniz. Le noyau du morphisme est donc invariant sous cette action. Ce noyau contient les \'el\'ements $(1\otimes f)(1\otimes g)-(1\otimes fg)$, $(f,g)\in \mathcal{O}_M(U)^2$. Il contient l'id\'eal $\mathcal{L}$. Le morphisme passe au quotient par cet id\'eal.
 
 Montrons que c'est un isomorphisme. Il est surjectif. Montrons l'injectivit\'e. Par respect de la r\`egle de Leibniz, un \'el\'ement de $\mathcal{O}_{RM}(U)$ s'\'ecrit 
 \[
 P=\displaystyle{\underset{i,\alpha,\ell}{\sum}}c_{i,\alpha,\ell}(\partial^{\alpha_1}\otimes y_{i_1})^{\ell_1}\ldots(\partial^{\alpha_n}\otimes y_{i_n})^{\ell_n}
 \]
  pour $c_{i,\alpha,\ell}\in\mathbb{C}$. Il est envoy\'e par $\phi$ sur $\displaystyle{\underset{i,\alpha,\ell}{\sum}}c_{i,\alpha,\ell}(y_{i_1}^{\alpha_1})^{\ell_1}\ldots (y_{i_n}^{\alpha_n})^{\ell_n}$. Si $P$ est dans le noyau de $\phi$ alors pour tout $i,\alpha,\ell$, $c_{i,\alpha,\ell}=0$ puisque la famille $\{(y_{i_1}^{\alpha_1})^{\ell_1}\ldots (y_{i_n}^{\alpha_n})^{\ell_n}\}$ est libre. Donc $P=0$ et $\phi$ est injectif. Apr\`es inversion du d\'eterminant jacobien dans les anneaux de d\'epart et d'arriv\'ee, nous obtenons l'isomorphisme escompt\'e.

La restriction de $\phi$ \`a $\mathcal{O}_{R_kM}(U)$ donne la deuxi\`eme assertion.
\end{proof}

\subsubsection{Structure de fibr\'e principal}

Notons $\Gamma$ le groupe des applications formelles inversibles de $(\mathbb{C}^m,0)$ dans $(\mathbb{C}^m,0)$. Pour $k\in\mathbb{N},$ notons $\Gamma_k$ le groupe des jets d'ordre $k$ des \'el\'ements de $\Gamma$. Le groupe $\Gamma$, resp. $\Gamma_k$, agit sur $R(M/S)$, resp. $R_k(M/S)$, par composition \`a la source. Ces actions:
\[
\begin{array}{ccc}
 			R(M/S)\times \Gamma &\longrightarrow & R(M/S)\\
			(r,\gamma)&\longmapsto & r\circ\gamma
\end{array}
\qquad \mbox{} \qquad
\begin{array}{cccc}
			R_k(M/S)\times \Gamma_k&\longrightarrow & R_k(M/S)\\
			(j_k(r),j_k(\gamma))&\longmapsto & j_k(r\circ\gamma)
\end{array}
\]
sont not\'ees $S_\gamma(r):=r\circ\gamma$ et $S_{j_k\gamma}(j_kr):=j_k(r\circ\gamma)$.

\begin{lemma}
Les espaces des $S$-rep\`eres d'ordre $k\in\mathbb{N}$ sont des fibr\'es principaux. Autrement dit, les applications
\[
\begin{array}{ccc}
			R_k(M/S)\times \Gamma_k&\longrightarrow & R_k(M/S)\times_M R_k(M/S)\\
			(j_k(r),j_k(\gamma))&\longmapsto & \left(j_k(r),j_k(r\circ\gamma)\right)
\end{array}
\]
sont des isomorphismes. Ces fibr\'es principaux sont localement triviaux pour la topologie de Zariski.
\end{lemma}

\begin{proof}
Les applications r\'eciproques sont donn\'ees par
\[
\begin{array}{ccc}
			 R_k(M/S)\times_M R_k(M/S)&\longrightarrow &R_k(M/S)\times \Gamma_k\\
			 \left(j_k(r_1),j_k(r_2)\right)&\longmapsto &(j_k(r_1),j_k(r_1^{-1}\circ r_2))
\end{array}
\]
Pla\c{c}ons-nous dans les coordonn\'ees donn\'ees par l'ouvert $U\subset M$ du lemme (\ref{lemreperelocglob}) et reprenons les notations de ce lemme. Gr\^{a}ce aux formules de Faa Di Bruno, nous savons que les applications induites sur les anneaux de coordonn\'ees sont polynomiales en les $y_i^\alpha$ pour $1\leq i\leq m,\ \alpha\in\mathbb{N}^m$. La vari\'et\'e $M$ peut \^{e}tre recouverte par de tels ouverts. Ceci conclut la preuve du lemme.
\end{proof}

L'action de $\Gamma$ sur $\mathcal{O}_{R(M/S)}$ peut se d\'ecrire directement en suivant son action sur $\mathbb{C}[\partial_1,\ldots,\partial_m]$. Pour tout $\gamma\in\Gamma$:

\begin{equation}
\label{actiondiff}
\begin{array}{ccc}
\gamma^*:\mathbb{C}[\partial_1,\ldots,\partial_m]&\longrightarrow &\mathbb{C}[\partial_1,\ldots,\partial_m]\\
P & \longmapsto &\gamma^*P\vert_{\e=0}
\end{array}
\end{equation}


\subsection{Le groupo\"{i}de des $S$-automorphismes associ\'e}\label{secapplpartiel}

\subsubsection{D\'efinition du groupo\"{i}de}

\begin{definition}
Un $S$-automorphisme de $p_1\in M$ vers $p_2\in M$, est une application formelle inversible $(M_{\rho(p_1)},p_1)\rightarrow(M_{\rho(p_2)},p_2)$.
Un $S$-automorphisme d'ordre $k$ est le jet d'ordre $k$ d'un $S$-automorphisme.
\end{definition}

Soit $U\subset M\times M$ un voisinage ouvert de $(p_1,p_2)\in M\times M$ qui soit un rev\^{e}tement non ramifi\'e d'un ouvert de $\mathbb{C}^{2(m+d)}$ et qui fasse commuter les diagrammes:
\[
\xymatrix{
    U\ar[r] \ar[d]_{pr_1/pr_2} & \mathbb{C}^{m+d}\times\mathbb{C}^{m+d}\ar[d]^{pr_1/pr_2} \\
    M\ar[r]\ar[d]&\mathbb{C}^{m+d}\ar[d]\\
    S\ar[r]&\mathbb{C}^d
    }
\]
o\`{u} $pr_1$ et $pr_2$ sont les projections respectives sur les premiers et derniers facteurs.
 Il vient un syst\`eme de coordonn\'ees $p_1=(p_1',q_1)\in \mathbb{C}^m\times\mathbb{C}^d,\ p_2=(p_2',q_2)\in \mathbb{C}^m\times\mathbb{C}^d$ dans lequel un $S$-automorphisme s'\'ecrit: 
\[
\phi(p_1'+(\e_1,\ldots,\e_m),q_1)=\left(\underset{\alpha\in \mathbb{N}^m}{\sum}\phi_1^{\alpha}\frac{\epsilon^{\alpha}}{\alpha !},\ldots,\underset{\alpha\in \mathbb{N}^m}{\sum}\phi_m^{\alpha}\frac{\epsilon^{\alpha}}{\alpha !},q_2\right)
\]
o\`{u}, pour $1\leq i\leq m$, $\alpha\in \mathbb{N}^m$, $\phi_i^{\alpha}\in \mathbb{C}$. La s\'erie tronqu\'ee \`a l'ordre $k$ est un $S$-rep\`ere d'ordre $k$.

\begin{definition}
L'espace des $S$-automorphismes est $Aut(M/S)$ et l'espace des $S$-automorphismes d'ordre $k\in\mathbb{N}$ est $Aut_k(M/S)$.  
\end{definition}

Par d\'efinition, l'espace des $S$-automorphismes est la limite projective des espaces des $S$-automorphismes d'ordre $k\in\mathbb{N}$.

Ces espaces sont des groupo\"{i}des o\`{u}:
\begin{itemize}
\item la base est $M$
\item les applications source et but sont:
\[
\begin{array}{cccc}
s:&Aut(M/S)&\longrightarrow &M\\
&(M_{\rho(p_1)},p_1)\rightarrow(M_{\rho(p_2)},p_2)&\longmapsto &p_1
\end{array}
\]
\[
\begin{array}{cccc}
t:&Aut(M/S)&\longrightarrow &M\\
&(M_{\rho(p_1)},p_1)\rightarrow(M_{\rho(p_2)},p_2)&\longmapsto &p_2
\end{array}
\]
\item l'identit\'e est
\[
\begin{array}{cccc}
e:&M&\longrightarrow &Aut(M/S)\\
&p&\longmapsto &id:(M_{\rho(p)},p)\rightarrow(M_{\rho(p)},p)
\end{array}
\]
\item en notant $Aut(M/S)\underset{{}^s\!M^t}{\times}Aut(M/S)$ le produit fibr\'e de $Aut(M/S)\overset{s}{\rightarrow }M$ et $Aut(M/S)\overset{t}{\rightarrow}M$, la composition partielle est 
\[
\begin{array}{cccc}
m :&Aut(M/S)\underset{{}^s\!M^t}{\times}Aut(M/S)&\rightarrow &Aut(M/S)\\
&(\phi_1,\phi_2)&\mapsto &\phi_1\circ \phi_2
\end{array}
\]

\item l'inverse est
\[
\begin{array}{cccc}
inv :&Aut(M/S)&\rightarrow &Aut(M/S)\\
&\phi&\mapsto &\phi^{-1}
\end{array}
\]
\end{itemize}

\subsubsection{Structure alg\'ebrique}\label{structalg}

La description locale de la structure alg\'ebrique de $Aut(M/S)$ se fait de la mani\`ere suivante.

Reprenons un ouvert $U\subset M\times M$ comme ci-dessus avec pour coordonn\'ees $x_1,\ldots,x_m,z_1,\ldots,z_d $, $y_1,\ldots y_m,t_1,\ldots,t_d$. L'ensemble des $S$-automorphismes d'ordre $k\in\mathbb{N}$ sur $U$ est muni d'une structure de vari\'et\'e affine dont l'anneau de coordonn\'ees est 
\[
\mathcal{O}_{M\times M}(U)[y_i^{\alpha}\ 1\leq i\leq m\ \alpha\in \mathbb{N}^m\ 1\leq\vert\alpha\vert\leq k][1/det(y_i^{1_j})]
\]
o\`{u} $y_i^\alpha$ la fonction d\'efinie par $y_i^\alpha (\phi) = \frac{\partial^\alpha \phi_i}{\partial x^\alpha}(p'_1)$. L'anneau de coordonn\'ees de la limite projective $Aut(M/S)\vert_U$ est 
\[
\mathcal{O}_{M\times M}(U)[y_i^{\alpha}\ 1\leq i\leq m\ \alpha\in \mathbb{N}^m][1/det(y_i^{1_j})]
\]
Notons $\mathcal{D}_{M/S}$ le $\mathcal{O}_M$-module des op\'erateurs diff\'erentiels sur $M$ $\mathcal{O}_S$-lin\'eaire et $pr_1:M\times M\rightarrow M$ la projection sur le premier facteur. Dans les coordonn\'ees que nous nous sommes donn\'ees, $\mathcal{D}_{M/S}(pr_1(U))$ est engendr\'e par $\mathcal{O}_M(pr_1(U))$ et les d\'erivations $\frac{\partial}{\partial x_1},\ldots,\frac{\partial}{\partial x_m}$. Ces d\'erivations agissent sur l'anneau de coordonn\'ees de $Aut(M/S)\vert_U$ de la mani\`ere suivante: $\frac{\partial}{\partial x_i}x_j=\delta_{ij},\ \frac{\partial}{\partial x_i}z_j=0,\ \frac{\partial}{\partial x_i}y_j^\alpha=y_j^{\alpha+1_j},\ \frac{\partial}{\partial x_i}t_j=0$. Ceci donne une structure de $\mathcal{D}_{M/S}(pr_1(U))$-alg\`ebre diff\'erentielle \`a cet anneau. 

Passons \`a la description globale de la structure alg\'ebrique de $Aut(M/S)$. Cette structure est enti\`erement d\'etermin\'ee par son image directe sur $M\times M$ par $Aut(M/S)\overset{s\times t}{\rightarrow} M\times M$. Par abus, nous parlerons indiff\'eremment de la structure alg\'ebrique de $Aut(M/S)$ et de son image directe sur $M\times M$.

Notons $Jac$ le faisceau en id\'eaux engendr\'e par $\det(y_i^{1_j})$ o\`{u} $y_1,\ldots,y_m$ sont des coordonn\'ees locales des fibres de $\rho\circ pr_2$ sur $U\subset M\times M$. 
 
\begin{definition}
Le faisceau $\mathcal{O}_{Aut(M/S)}$ des coordonn\'ees de l'espace des $S$-automorphismes est:
\[
\left(Sym(pr_1^*\mathcal{D}_{M/S}\otimes pr_2^*\mathcal{O}_M)/\mathcal{L}\right)[1/Jac]
\]
o\`{u} 
\begin{itemize}
\item $\mathcal{D}_{M/S}$ est le $\mathcal{O}_M$-module des op\'erateurs diff\'erentiels sur $M$ $\mathcal{O}_S$-lin\'eaire
\item $pr_1$ et $pr_2$ sont les projections sur le premier et le deuxi\`eme facteur de $M\times M$ 
\item $pr_1^*\mathcal{D}_{M/S}\otimes pr_2^*\mathcal{O}_M$ est un produit tensoriel d'espace vectoriel,
\item $Sym(pr_1^*\mathcal{D}_{M/S}\otimes pr_2^*\mathcal{O}_M)$ est l'alg\`ebre sym\'etrique engendr\'ee,
\item cette alg\`ebre a une structure de $\mathcal{D}_{M/S}$-alg\`ebre diff\'erentielle,
\item elle est quotient\'ee par l'id\'eal diff\'erentiel $\mathcal{L}$ engendr\'e par le noyau de $Sym(pr_1^*\mathcal{O}_M\otimes pr_2^*\mathcal{O}_M)\rightarrow \mathcal{O}_{M\times M}$,
\item elle est enfin localis\'ee en $Jac$
\end{itemize}
Le faisceau $\mathcal{O}_{Aut_k(M/S)}$ des coordonn\'ees de l'espace des $S$-automorphismes d'ordre $k\in\mathbb{N}$ est:
\[
\left(Sym(pr_1^*\mathcal{D}_{M/S}^{\leq k}\otimes pr_2^*\mathcal{O}_M)/(\mathcal{L}\cap Sym(pr_1^*\mathcal{D}_{M/S}^{\leq k}\otimes pr_2^*\mathcal{O}_M))\right)[1/Jac]
\]
o\`{u} $\mathcal{D}_{M/S}^{\leq k}$ d\'esigne l'ensemble des op\'erateurs diff\'erentiels d'ordre au plus $k$.
\end{definition}

L'espace des $S$-automorphismes d'ordre $k\in\mathbb{N}$ devient un groupo\"{i}de alg\'ebrique. Nous retrouvons la structure locale. Soit $U\subset M\times_SM$ un rev\^{e}tement non-ramifi\'e d'un ouvert de $\mathbb{C}^{2(m+d)}$ comme ci-dessus de coordonn\'ees  $x_1,\ldots,x_m,z_1,\ldots,z_d,y_1,\ldots,y_m,t_1,\ldots,t_d$:

\begin{lemma}\label{lemautloc}
Nous avons un isomorphisme:
\[
\mathcal{O}_{Aut(M/S)}(U)\simeq\mathcal{O}_{M\times M}(U)[y_i
^{\alpha}\mbox{, } 1\leq i\leq n\mbox{, }\alpha\in \mathbb{N}^m\mbox{, }1\leq\vert\alpha\vert][1/det(y_i^{1_j})]
\]
 qui est compatible avec l'action des op\'erateurs diff\'erentiels. Sa restriction induit:
 \[
 \mathcal{O}_{Aut_k(M/S)}(U)\simeq\mathcal{O}_{M\times M}(U)[y_i^{\alpha}\mbox{, } 1\leq i\leq m\mbox{, }\alpha\in \mathbb{N}^m\mbox{, }1\leq\vert\alpha\vert\leq k][1/det(y_i^{1_j})]
 \]
\end{lemma}

\begin{proof}
La preuve est la m\^{e}me que celle du lemme (\ref{lemreperelocglob}).
Commen\c{c}ons par d\'efinir un morphisme d'espace vectoriel \`a partir de l'injection canonique:
\[
\begin{array}{cccc}
\phi :&pr_1^*\mathcal{O}_{M}\otimes pr_2^*\mathcal{O}_M(U)&\longrightarrow &\mathcal{O}_{M\times M}(U)[y_i^{\alpha}\mbox{, } 1\leq i\leq m\mbox{, }\alpha\in \mathbb{N}^m\mbox{, }1\leq\vert\alpha\vert]\\
&f\otimes g&\longmapsto &f\otimes g 
\end{array}
\]
Nous allons le prolonger en utilisant la structure de $\mathcal{D}_{M/S}(pr_1(U))$- module puis en utilisant la structure d'anneau de l'espace d'arriv\'e. 
Puisque $\mathcal{D}_{M/S}(pr_1(U))$ agit sur $pr_1^*\mathcal{O}_{M}\otimes pr_2^*\mathcal{O}_M(U)$, cette application se prolonge en un morphisme de $\mathcal{D}_{M/S}(pr_1(U))$-module 
\[
pr_1^*\mathcal{D}_{M/S}\otimes pr_2^*\mathcal{O}_M(U)\to \mathcal{O}_{M\times M}(U)[y_i^{\alpha}\mbox{, } 1\leq i\leq m\mbox{, }\alpha\in \mathbb{N}^m\mbox{, }1\leq\vert\alpha\vert]
\]
Nous en d\'eduisons un morphisme d'alg\`ebre
\[
Sym(pr_1^*\mathcal{D}_{M/S}\otimes pr_2^*\mathcal{O}_M)(U)\to \mathcal{O}_{M\times M}(U)[y_i^{\alpha}\mbox{, } 1\leq i\leq m\mbox{, }\alpha\in \mathbb{N}^m\mbox{, }1\leq\vert\alpha\vert]
\]
%
 Il est encore compatible aux actions des op\'erateurs diff\'erentiels de $M$ car elles doivent toutes deux v\'erifier la r\`egle de Leibniz. Le noyau du morphisme est donc invariant sous cette action. 
 
 On v\'erifie comme en \ref{lemreperelocglob} que le noyau de ce morphisme est l'id\'eal diff\'erentiel engendr\'e par le noyau de $Sym(pr_1^*\mathcal{O}_M\otimes pr_2^*\mathcal{O}_M)\rightarrow \mathcal{O}_{M\times M}$. 


La restriction de $\phi$ \`a $\mathcal{O}_{Aut_k(M/S)}$ donne la deuxi\`eme assertion. 
\end{proof}

Pour $k\in\mathbb{N}$, le groupe $\Gamma_k$ agit diagonalement sur le produit $R_k(M/S)\times R_k(M/S)$. Le quotient $\left(R_k(M/S)\times R_k(M/S)\right)/\Gamma_k$ h\'erite de la structure alg\'ebrique de $Aut_k(M/S)$. Ceci est donn\'e par le lemme:

\begin{lemma}\label{quotientalg}
Soit $k\in\mathbb{N}$. L'application 
\[
\begin{array}{cccc}
\Phi_k : & R_k(M/S)\times R_k(M/S)&\longrightarrow & Aut_k(M/S)\\
&(j_k(r),j_k(s))&\longmapsto & j_k(s\circ r^{-1})
\end{array}
\]
est un morphisme surjectif. Les fibres sont les orbites sous l'action diagonale du groupe $\Gamma_k$.
\end{lemma}

\begin{proof}
Nous savons que l'application est un morphisme gr\^{a}ce aux formules de Faa Di Bruno. Ce morphisme est surjectif puisque pour $\phi\in Aut(M/S)$, $s\in R(M/S)$, $j_k(\phi)=j_k(s\circ (s^{-1}\circ\phi))$.

Les fibres contiennent les orbites:  pour $(s,r)\in R(M/S)^2$, $\gamma\in\Gamma$, $j_k(s\circ r^{-1})=j_k(s\circ\gamma\circ\gamma^{-1}\circ r^{-1})$. 
Les orbites contiennent les fibres: soit $(r,r',s,s')\in R(M/S)^4$ tels que $j_k(s\circ r^{-1})=j_k(s'\circ r'^{-1})$. Posons $j_k(\gamma)=j_k(s^{-1}\circ s')=j_k(r^{-1}\circ r')$. Nous avons $j_k(r')=j_k(r\circ\gamma)$ et $j_k(s')=j_k(s\circ\gamma)$.
\end{proof}

Sur le produit $R_k(M/S)\times R_k(M/S)$ sont d\'efinies les applications
\begin{itemize}
\item identit\'e:
\[
\begin{array}{cccc}
e:&R_k(M/S)&\longrightarrow &R_k(M/S)\times R_k(M/S)\\
&j_k(r)&\longmapsto &(j_k(r),j_k(r))
\end{array}
\]
\item composition partielle:
\[
\begin{array}{cccc}
m :&R_k(M/S)\times R_k(M/S)\underset{{}^{pr_1}\!M^{pr_2}}{\times}R_k(M/S)\times R_k(M/S)&\longrightarrow &R_k(M/S)\times R_k(M/S)\\
&(j_k(r_1),j_k(r_2)),(j_k(r_3),j_k(r_1))&\longmapsto &(j_k(r_3),j_k(r_2))
\end{array}
\]

\item inverse
\[
\begin{array}{cccc}
inv :&R_k(M/S)\times R_k(M/S)&\longrightarrow &R_k(M/S)\times R_k(M/S)\\
&(j_k(r_1),j_k(r_2))&\longmapsto &(j_k(r_2),j_k(r_1))
\end{array}
\]
\end{itemize}
Le morphisme $\Phi_k$ donne une correspondance entre les relations d'\'equivalence de $R_k(M/S)\times R_k(M/S)$ (\textit{i.e.} les vari\'et\'es stables par ces trois applications) qui sont stables sous l'action diagonale du groupe $\Gamma_k$ et les sous-groupo\"{i}des alg\'ebriques de $Aut_k(M/S)$.

Nous terminons cette sous-section en comparant les structures diff\'erentielles du produit $R(M/S)\times R(M/S)$ et $Aut(M/S)$. Notons $\Phi$ l'application quotient envoyant l'un sur l'autre. 
\begin{lemma}\label{lemdiff}
Un $I$ un id\'eal de $\mathcal{O}_{Aut(M/S)}$ est $\mathcal{D}_{M/S}$-invariant si et seulement si $\Phi^*I\subset \mathcal{O}_{R(M/S)\times R(M/S)}$ est $(\partial_i\otimes 1+1\otimes\partial_i)$-invariant pour tout $1\leq i\leq m$.
\end{lemma}

\subsection{Prolongements de champs de vecteurs}

Soit $X$ un champ de vecteurs rationnel sur la vari\'et\'e $M$ tangent aux fibres de $\rho$. Ce champ se prolonge naturellement aux espaces des $S$-rep\`eres puis aux espaces des $S$-automorphismes.

\subsubsection{Prolongements aux espaces des $S$-rep\`eres}

D\'ecrivons deux constructions du prolongement du champ $X$.

Une description alg\'ebrique consiste \`a \'etendre la d\'erivation $X$ de $\mathcal{O}_M$ sur $\mathcal{O}_{R(M/S)}$. La d\'erivation $X:\mathcal{O}_M\rightarrow \mathcal{O}_M$ s'\'etend en un endomorphisme $1\otimes X:\mathbb{C}[\partial_1,\ldots,\partial_m]\otimes \mathcal{O}_M\rightarrow \mathbb{C}[\partial_1,\ldots,\partial_m]\otimes \mathcal{O}_M$. Puis en une d\'erivation $1\otimes X:Sym(\mathbb{C}[\partial_1,\ldots,\partial_m]\otimes \mathcal{O}_M)\rightarrow Sym(\mathbb{C}[\partial_1,\ldots,\partial_m]\otimes \mathcal{O}_M)$. On v\'erifie que $1\otimes X(\mathcal{L})\subset \mathcal{L}$, ce qui induit une d\'erivation $RX$ sur $\mathcal{O}_{R(M/S)}$. L'endomorphisme $1\otimes X:\mathbb{C}[\partial_1,\ldots,\partial_m]\otimes \mathcal{O}_M\rightarrow \mathbb{C}[\partial_1,\ldots,\partial_m]\otimes \mathcal{O}_M$ pr\'eserve la filtration par l'ordre des op\'erateurs diff\'erentiels. Ceci implique que $RX$ pr\'eserve $\mathcal{O}_{R_k(M/S)}$. Sa restriction \`a cette sous-alg\`ebre est not\'ee $R_kX$.
Nous voyons dans cette construction que les prolongements commutent \`a l'action de $\mathbb{C}[\partial_1,\ldots,\partial_m]$ d\'efinie dans la sous-section (\ref{ssecalgrep}), \textit{i.e.}  v\'erifient les \'egalit\'es 
\begin{equation}\label{relationderivtot}
R_{k+1}X\circ \partial_i-\partial_i\circ R_kX=0
\end{equation}

Une description analytique consiste \`a prendre le champ tangent aux prolongements des flots du champ $X$ sur $R(M/S)$. Les flots du champ $X$ d\'efinissent des biholomorphismes dans les fibres de $\rho$ sur leurs ouverts de d\'efinition. Un flot au temps $\tau\in\mathbb{C}$ sera not\'e $exp(\tau X)$ ind\'ependamment de son ouvert de d\'efinition. Leurs jets d\'efinissent des $S$-automorphismes. Or l'espace des $S$-automorphismes agit sur l'espace des $S$-rep\`eres par composition au but:
\[
\begin{array}{ccc}
Aut(M/S)\underset{{}^s\!M}{\times} R(M/S)&\longrightarrow & R(M/S)\\
(\phi,r)&\longmapsto & \phi\circ r
\end{array}
\]
 Les prolongements $RX$ et $R_kX$ de $X$ sont d\'efinis sur $R(M/S)$ et $R_k(M/S)$ gr\^{a}ce \`a cette action. Pour $r\in R(M/S)$ tel que $r(0)$ soit dans le domaine de d\'efinition de $X$: 
\[
RX\vert_{r} :=\left.\frac{d}{d\tau}exp(\tau X)\circ r\right\vert_{\tau=0}
\]
\[
 R_kX\vert_{j_kr} :=\left.\frac{d}{d\tau}j_k(exp(\tau X)\circ r)\right\vert_{\tau=0}
\]

\begin{lemma} 
Les deux mani\`eres de prolonger le champ $X$ en $RX$ et en $R_kX$ aux espaces des $S$-rep\`eres $R(M/S)$ et $R_k(M/S)$ d\'ecrites ci-dessus co\"{i}ncident. 
\end{lemma}

\begin{proof}
D'apr\`es \cite[pp 117--135]{Olver} la construction analytique v\'erifient aussi les relations (\ref{relationderivtot}) et ces relations d\'efinissent compl\`etement le champ de vecteurs.
\end{proof}

L'action de $\Gamma$ sur l'anneau de coordonn\'ees de $R(M/S)$ provient de son action sur le facteur de gauche dans $\mathbb{C}[\partial_1,\ldots,\partial_m]\otimes \mathcal{O}_M$.  (c.f. \ref{actiondiff}). La d\'erivation $RX$ est induite par un endomorphisme agissant sur le membre de droite du m\^{e}me produit tensoriel. Ces deux action commutent. Autrement dit, pour tout $\gamma\in\Gamma$,

 \begin{equation}\label{equaGinv}
 S_\gamma^*RX=RX\qquad\mbox{et}\qquad \left(S_{j_k(\gamma)}\right)^*R_kX=R_kX 
 \end{equation}

\subsubsection{Prolongements aux espaces des $S$-automorphismes}

Le champ de vecteurs $X$ peut se prolonger aux espaces des $S$-automorphismes par la source ou par le but. Ces prolongements s'expriment \`a l'aide du morphisme $\Phi_k:R_k(M/S)\times R_k(M/S)\rightarrow Aut_k(M/S)$ donn\'e dans le lemme (\ref{quotientalg}):

\begin{definition}
Les prolongements du champ $X$ aux espaces des $S$-automorphismes par la source et par le but sont respectivement:
\[
R_k^sX=d\Phi_k.(R_kX + 0)\qquad
R_k^tX=d\Phi_k.(0 + R_kX)
\]
\end{definition}

Ces champs sont bien d\'efinis puisque les champs $R_kX + 0$ et $0+R_kX $ sont invariants sous l'action diagonale de $\Gamma_k$. 

\section{Le groupo\"ide de Galois}\label{secgrouppart}

  \subsection{Trois d\'efinitions d'un groupo\"{i}de}\label{soussec3def}

Nous pr\'esentons ici trois d\'efinitions d'une sous-vari\'et\'e de $Aut_k(M/S)$ associ\'e au champ $X$. Cette sous-vari\'et\'e est un sous-groupo\"{i}de au-dessus d'un ouvert de $M$. Par abus de langage, nous appellerons sous-groupo\"{i}de ce type de sous-vari\'et\'e. Il est commode d'utiliser la correspondance donn\'ee par le morphisme $\Phi_k$ du lemme (\ref{quotientalg}) entre les relations d'\'equivalence de $R_k(M/S)\times R_k(M/S)$ stables sous l'action diagonale du groupe $\Gamma_k$ et sous-groupo\"{i}des de $Aut_k(M/S)$. Par ailleurs, il sera r\'eguli\`erement n\'ecessaire de se restreindre \`a des ouverts de la base $M$. Si $U$ est un tel ouvert, si $\mathcal{G}$ est une sous-vari\'et\'e de $Aut_k(M/S)$ et si $\mathcal{R}$ est une sous-vari\'et\'e de $R_k(M/S)\times R_k(M/S)$, alors les restrictions respectives au-dessus de cet ouvert seront not\'ees $\mathcal{G}\vert_{U\times U}$ et $\mathcal{R}\vert_{U\times U}$, autrement dit : $\mathcal{G}\vert_{U\times U}:=\mathcal{G}\cap (s\times t)^{-1}(U\times U)$ et $\mathcal{R}\vert_{U\times U}:=\mathcal{R}\cap (pr_1\times pr_2)^{-1}(U\times U)$.

  \subsubsection{D\'efinition alg\'ebrique}
  
 Voici la d\'efinition du groupo\"{i}de en tant que plus petit groupo\"{i}de alg\'ebrique contenant le flot de $X$:
 
\begin{definition}
Le groupo\"{i}de $\mathcal{G}_k^a(X)$, est la plus petite sous-vari\'et\'e de $Aut_k(M/S)$ telle qu'il existe un ouvert $U^a\subset M$ v\'erifiant :

\begin{enumerate}
\item $R_k^tX\subset T\mathcal{G}_k^a(X)$ 
\item la restriction $\mathcal{G}_k^a(X)\vert_{U^a\times U^a}$ est un sous-groupo\"{i}de de $Aut_k(M/S)\vert_{U^a\times U^a}$
\end{enumerate}

\end{definition}

Afin de voir que cette vari\'et\'e est bien d\'efinie, donnons la d\'efinition d'une sous-vari\'et\'e minimale tangente de $R_k(M/S)\times R_k(M/S)$ qui soit une relation d'\'equivalence au-dessus d'un ouvert. Nous allons montrer que l'existence de l'une est \'equivalente \`a l'existence de l'autre, puis que cette deuxi\`eme vari\'et\'e est bien d\'efinie.

\begin{definition} \label{Ra}
 La vari\'et\'e $\mathcal{R}_k^a(X)$ est la plus petite sous-vari\'et\'e de $R_k(M/S)\times R_k(M/S)$ telle qu'il existe un ouvert $U^a\subset B$ v\'erifiant :
\begin{enumerate}
\item $0+R_kX\subset T\mathcal{R}_k^a(X)$
\item la restriction $\mathcal{R}_k^a(X)\vert_{U^a\times U^a}$ est une relation d'\'equivalence de $R_k(U^a/S)\times R_k(U^a/S)$ 
\end{enumerate}
\end{definition}

\begin{lemma}\label{corresRa}
En supposant que les vari\'et\'es $\mathcal{G}_k^a(X)$ et $\mathcal{R}_k^a(X)$ existent, nous avons les \'egalit\'es: $\mathcal{G}_k^a(X)=\Phi_k(\mathcal{R}_k^a(X))$ et $\Phi_k^{-1}(\mathcal{G}_k^a(X))=\mathcal{R}_k^a(X)$.
\end{lemma}

\begin{proof}
Gr\^{a}ce \`a la correspondance donn\'ee par $\Phi_k$, il suffit de montrer que la vari\'et\'e $\mathcal{R}_k^a(X)$ est stable sous l'action diagonale du groupe $\Gamma_k$. Soit $j_k(\gamma)\in \Gamma_k$. L'action diagonale de $j_k(\gamma)$ est not\'ee $\Delta_{j_k(\gamma)}$.

Montrons que la sous-vari\'et\'e $\Delta_{j_k(\gamma)}(\mathcal{R}_k^a(X))$ v\'erifie la condition (1) de la d\'efinition (\ref{Ra}). Le champ $R_kX$ est invariant sous l'action du groupe $\Gamma_k$. Nous avons l'inclusion 
\[0+R_kX=d\Delta_{j_k(\gamma)}.(0+R_kX)\subset d\Delta_{j_k(\gamma)}.T\mathcal{R}_k^a(X)=T(\Delta_{j_k(\gamma)}(\mathcal{R}_k^a(X)))
\]
Montrons que la sous-vari\'et\'e $\Delta_{j_k(\gamma)}(\mathcal{R}_k^a(X))$ v\'erifie la condition (2). Dans la sous-section (\ref{structalg}), nous avions d\'efini l'identit\'e $e$, la composition partielle $m$ et l'inversion $inv$ sur le produit $R_k(M/S)\times R_k(M/S)$. Montrons la stabilit\'e de la sous-vari\'et\'e sous ces trois applications au-dessus de l'ouvert $U^a$:
\begin{align*}
e(U^a)&=\Delta_{j_k(\gamma)}(e(U^a))\\
&\subset \Delta_{j_k(\gamma)}(\mathcal{R}_k^a(X)\vert_{U^a\times U^a})
\end{align*}
L'inversion et la composition partielle commutent avec l'action diagonale et stabilisent $\mathcal{R}_k^a(X)\vert_{U^a\times U^a}$, donc 
\begin{align*}
inv\circ \Delta_{j_k(\gamma)}(\mathcal{R}_k^a(X)\vert_{U^a\times U^a})&=\Delta_{j_k(\gamma)}\circ inv(\mathcal{R}_k^a(X)\vert_{U^a\times U^a})\\
&\subset \Delta_{j_k(\gamma)}(\mathcal{R}_k^a(X)\vert_{U^a\times U^a})
\end{align*} et 
\begin{multline*}
m\left(\Delta_{j_k(\gamma)}(\mathcal{R}_k^a(X)\vert_{U^a\times U^a})\underset{{}^s\!R_k(U^a/S)^t}{\times} \Delta_{j_k(\gamma)}(\mathcal{R}_k^a(X)\vert_{U^a\times U^a})\right) \\
 =\Delta_{j_k(\gamma)}\circ m\left(\mathcal{R}_k^a(X)\vert_{U^a\times U^a}\underset{{}^s\!R_k(U^a/S)^t}{\times}\mathcal{R}_k^a(X)\vert_{U^a\times U^a}\right)\\
\subset \Delta_{j_k(\gamma)}(\mathcal{R}_k^a(X)\vert_{U^a\times U^a})
\end{multline*}

En appelant la minimalit\'e de $\mathcal{R}_k^a(X)$, il vient l'inclusion $\mathcal{R}_k^a(X)\subset \Delta_{j_k(\gamma)}(\mathcal{R}_k^a(X))$. Pour tout ${j_k(\gamma)}\in \Gamma_k$, nous avons l'inclusion souhait\'ee $\Delta_{{j_k(\gamma}^{-1})}(\mathcal{R}_k^a(X))\subset \mathcal{R}_k^a(X)$.
\end{proof}


\begin{lemma}\label{lemRa}
La sous-vari\'et\'e $\mathcal{R}_k^a(X)$ est bien d\'efinie.
\end{lemma} 
 
\begin{proof} 
Par le lemme (\ref{lemreperelocglob}), il existe un ouvert affine $M^\circ\subset M$ tel que le fibr\'e des $S$-rep\`eres $R_k(M^\circ/S)$ soit aussi affine. Cet ouvert peut \^{e}tre pris de telle sorte que le prolongement $R_kX$ soit bien d\'efini sur $R_k(M^\circ/S)$.
Notons $K:=\mathcal{O}({R_k(M^\circ/S)})\underset{\mathcal{O}({M^\circ})}{\otimes}\mathbb{C}(M)$ l'anneau des fonctions r\'eguli\`eres sur $R_k(M^\circ/S)$ localis\'e en les fonctions r\'eguli\`eres de l'ouvert affine $M^\circ$. Notons $L:=K\otimes K$ l'anneau des fonctions r\'eguli\`eres de $R_k(M^\circ/S)\times R_k(M^\circ/S)$ localis\'e en les fonctions r\'eguli\`eres de l'ouvert $M^\circ$ par la source et par le but. Regardons les id\'eaux $I$ de $L$ v\'erifiant : 
\begin{enumerate}
\item[(i)] L'id\'eal $\mathcal{I}(e(M^\circ))$ de $L$ associ\'e \`a l'identit\'e $e(M^\circ)$ de $R_k(M^\circ/S)\times R_k(M^\circ/S)$ contient $I$
\item[(ii)] $inv^{*}(I)\subset I$ et $m^{*}(I)\subset I\underset{{}^sK^t}{\otimes} 1+1\underset{{}^sK^t}{\otimes} I$ 
\item[(iii)]  $(0+R_kX)(I)\subset I$
\end{enumerate}

 Si $I_1$ et $I_2$ sont deux tels id\'eaux de $L$, c'est-\`a-dire qui v\'erifient (i), (ii) et (iii), alors l'id\'eal $I_1+I_2$ v\'erifie encore (i). De plus $inv^*$, $m^*$ et l'action d'un champ de vecteurs $X$ \'etant additifs, il vient que $I_1+I_2$ v\'erifie aussi (ii) et (iii). Par n{\oe}th\'erianit\'e, il existe un id\'eal maximal, \`a nouveau not\'e $I$, v\'erifiant les trois conditions ci-dessus. Notons $N :=\mathcal{V}(I\cap \mathcal{O}(R_k(M^\circ/S)\times R_k(M^\circ/S)))$ la sous-vari\'et\'e de $R_k(M^\circ/S)\times R_k(M^\circ/S)$ correspondant \`a l'id\'eal $I$.
 
 Montrons que son adh\'erence $\overline{N}$ est la sous-vari\'et\'e de $R_k(M/S)\times R_k(M/S)$ recherch\'ee, c'est-\`a-dire qu'elle v\'erifie (1), (2) et qu'elle est minimale.
 
  La stabilit\'e de l'id\'eal $I$ sous $inv^*$ et $m^*$ perdure en n'inversant qu'une seule fonction de $\mathcal{O}(M^\circ)$. L'id\'eal $I$ est engendr\'e par des \'el\'ements $f_1,\ldots,f_n$, qui 
  peuvent \^{e}tre pris dans l'anneau $\mathcal{O}(R_k(M^\circ/S)\times R_k(M^\circ/S))$. L'id\'eal $I$ v\'erifiant (ii), il existe des \'el\'ements $\alpha_{ij}$, $\beta_{ij}$, $\gamma_{ij}$, $\delta_{ij}$ dans l'anneau localis\'e $L$ tels que $m^*(f_{i})=\displaystyle{\sum_{j}}\alpha_{ij}f_j\underset{{}K}{\otimes} \beta_{ij}+\gamma_{ij}\underset{{}K}{\otimes} \delta_{ij}f_j$. Il existe un multiple commun $h\in \mathcal{O}(M^\circ)$ des d\'enominateurs de tous ces \'el\'ements, c'est-\`a-dire, pour tout $i,j$, les \'el\'ements $\alpha_{ij}$, $\beta_{ij}$, $\gamma_{ij}$, $\delta_{ij}$ sont contenus dans l'anneau 
  $\mathcal{O}(R_k(M^\circ/S)\times R_k(M^\circ/S))\left[\frac{1}{h}\otimes 1,1\otimes \frac{1}{h}\right]$. 
  Cet anneau a de plus la particularit\'e d'\^{e}tre stable sous les morphismes $inv^*$ et $m^*$ puisque l'anneau $\mathcal{O}(R_k(M^\circ/S)\times R_k(M^\circ/S))$ l'est et que $inv^*\left(\frac{1}{h} \otimes1\right)=1\otimes \frac{1}{h}$,
 $inv^*\left(1 \otimes \frac{1}{h}\right)=\frac{1}{h}\otimes1$,
 $m^*\left(\frac{1}{h} \otimes 1\right)=\left(\frac{1}{h} \otimes 1\right) \otimes\left(1 \otimes 1\right)$ et
 $m^*\left(1 \otimes \frac{1}{h}\right)=\left(1 \otimes 1\right) \otimes\left(1\otimes\frac{1}{h}\right)$.
  Il vient alors la stabilit\'e recherch\'ee : en posant $I'=I\cap \mathcal{O}(R_k(M^\circ/S)\times R_k(M^\circ/S))\left[\frac{1}{h}\otimes 1,1\otimes \frac{1}{h}\right]$, alors $inv^{*}(I')\subset I'$ et $m^{*}(I')\subset I'  \otimes 1+1\otimes  I'$. Autrement dit, la sous-vari\'et\'e $N$ est stable par inversion et par composition partielle au-dessus de l'ouvert $M^\circ-\{h=0\}$. De plus, l'inclusion (i): $I\cap \mathcal{O}(R_k(M^\circ/S)\times R_k(M^\circ/S))\subset \mathcal{I}(e(M^\circ))$ implique $e(M^\circ-\{h=0\})\subset N$. Donc son adh\'erence $\overline{N}$ dans $R_k(M/S)$ v\'erifie la condition (2) avec cet ouvert.
 
 La condition (iii) implique que le prolongement $0+R_kX$ est tangent \`a la sous-vari\'et\'e $N$. Cette condition est ferm\'ee. Il vient que le prolongement $0+R_kX$ est tangent \`a son adh\'erence, \textit{i.e.} $\overline{N}$ v\'erifie (1).
 
 Il reste enfin la minimalit\'e. Soit $N'$ une sous-vari\'et\'e de $R_k(M/S)\times R_k(M/S)$ v\'erifiant (1) et (2). Notons $O'$ l'ouvert de $M$ correspondant \`a la condition (2). Montrons que cette sous-vari\'et\'e contient $\overline{N}$. L'intersection $N''=N'\cap \overline{N}$ v\'erifie elle aussi ces deux conditions o\`{u} l'ouvert correspondant \`a la condition (2) est l'intersection $O'\cap (M^\circ-\{k=0\})$. Notons $N_{M^\circ}''$ la restriction de $N''$ \`a l'ouvert affine $R_k(M^\circ/S)\times R_k(M^\circ/S)$ et $J :=\mathcal{I}(N_{M^\circ}'')$ l'id\'eal associ\'e. L'inclusion des sous-vari\'et\'es $N_{M^\circ}''\subset N$ implique l'inclusion inverse des id\'eaux associ\'es $(I\cap \mathcal{O}(R_k(M^\circ/S)\times R_k(M^\circ/S)))\subset J$. Ainsi, 
 il vient $I=(I\cap \mathcal{O}(R_k(M^\circ/S)\times R_k(M^\circ/S)))L\subset JL$. Or, tout comme l'id\'eal $I$, l'id\'eal $JL$ de l'anneau localis\'e $L$ v\'erifie aussi les conditions (i), (ii) et (iii). La maximalit\'e de $I$ implique $I=JL$. Ceci 
 donne l'inclusion $J\subset (JL\cap \mathcal{O}(R_k(M^\circ/S)\times R_k(M^\circ/S)))=I\cap\mathcal{O}(R_k(M^\circ/S)\times R_k(M^\circ/S))$. En revenant aux sous-vari\'et\'es associ\'ees, il vient l'inclusion inverse $N\subset N_{M^\circ}''$. Puis en passant aux adh\'erences, il vient $\overline{N}\subset \overline{ N_{M^\circ}''}\subset N''$. Finalement, nous avons montr\'e l'\'egalit\'e $\overline{N}=N''$, ce qui donne l'inclusion $\overline{N}\subset N'$ escompt\'ee.
 \end{proof}

 \subsubsection{D\'efinition topologique}
 
 Voici la d\'efinition du groupo\"{i}de en tant que cl\^{o}ture de Zariski du flot de $X$.
 
\begin{definition}\label{Rb}
 Le groupo\"{i}de $\mathcal{G}_k^b(X)$ est la plus petite sous-vari\'et\'e de $Aut_k(M/S)$ qui est tangente au prolongement $R_k^tX$ et qui contient l'identit\'e. Autrement dit, en reprenant les notations de la d\'efinition \cite[definition 1.2]{Bonnet} $\mathcal{G}_k^b(X) :=V(R^t_k X,e(M))$.
\end{definition}

Le corollaire  \cite[corollary 2.6]{Bonnet}  nous dit que cette vari\'et\'e minimale tangente est bien d\'efinie et qu'elle est irr\'eductible. De plus, comme pr\'ec\'edemment, il lui correspond une sous-vari\'et\'e minimale tangente de $R_k(M/S)\times R_k(M/S)$.
 
\begin{definition}
La vari\'et\'e $\mathcal{R}_k^b(X)$, est la plus petite sous-vari\'et\'e de $R_k(M/S)\times R_k(M/S)$ qui est tangente au prolongement $0+R_kX$ et qui contient l'identit\'e. Autrement dit, 
$\mathcal{R}_k^b(X) :=V(0+R_kX,e(M))$.
\end{definition}
De la m\^{e}me mani\`ere que dans la preuve du lemme (\ref{lemRa}), nous pouvons montrer que cette vari\'et\'e est invariante sous l'action diagonale et correspond par $\Phi_k$ au groupo\"{i}de $\mathcal{G}_k^b(X)$:
\begin{lemma}\label{corresRb}
Les vari\'et\'es minimales tangentes v\'erifient $\mathcal{G}_k^b(X)=\Phi_k(\mathcal{R}_k^b(X))$ et $\Phi_k^{-1}(\mathcal{G}_k^b(X))=\mathcal{R}_k^b(X)$.
\end{lemma}

\subsubsection{D\'efinition via les int\'egrales premi\`eres}
\begin{definition}
Une fonction rationnelle $H\in \mathbb{C}(R_k(M/S))$ est une int\'egrale premi\`ere rationnelle du prolongement $R_kX$ si $R_kX.H=0$.

Le sous-corps de $\mathbb{C}(R_k(M/S))$ constitu\'e des int\'egrales premi\`eres rationnelles du prolongement $R_kX$ est not\'e $\mathbb{C}(R_k(M/S))^{R_kX}$.
\end{definition}
Dans cette sous-section, le groupo\"{i}de est l'ensemble des $S$-automorphismes qui pr\'eservent les niveaux des int\'egrales premi\`eres rationnelles de $R_kX$:

\begin{definition}\label{Rc}
Soient $(H_i)_{1\leq i\leq n}$ une famille g\'en\'eratrice de $\mathbb{C}(R_k(M/S))^{R_kX}$ et $R_k(M/S)^\circ$ un ouvert de $R_k(M/S)$ sur lequel les applications $H_i$ sont bien d\'efinies. Le groupo\"{i}de $\mathcal{G}_k^c(X)$ est la cl\^{o}ture de Zariski de l'ensemble $\{\phi\in \Phi_k(R_k(M/S)^\circ\times R_k(M/S)^\circ)\ \vert\ \forall i,\ H_i\circ\phi=H_i\}$.
\end{definition}

Cette d\'efinition du groupo\"{i}de demande de faire un choix parmi les int\'egrales premi\`eres du prolongement du champ $X$. Nous allons d\'efinir une sous-vari\'et\'e de $R_k(M/S)\times R_k(M/S)$ qui lui correspond. Nous allons montrer qu'elle est intrins\`eque au champ $X$, qu'elle est invariante sous l'action diagonale du groupe $\Gamma_k$ et que son image par $\Phi_k$ est le groupo\"{i}de. Commen\c{c}ons par rappeler la d\'efinition d'int\'egrale premi\`ere maximale donn\'ee en annexe :

\begin{definition}
Soient $N$ une vari\'et\'e irr\'eductible et $\pi :R_k(M/S)\dashrightarrow N$ une application rationnelle dominante. L'application $\pi$ est une int\'egrale premi\`ere maximale de $R_kX$ si $\pi^*\mathbb{C}(N)=\mathbb{C}(R_k(M/S))^{R_kX}$.
\end{definition}
Voici la d\'efinition de la sous-vari\'et\'e dont nous avons parl\'e :
\begin{definition}\label{defRc}
Soient $\pi :R_k(M/S)\dashrightarrow N$ une int\'egrale premi\`ere maximale du prolongement $R_kX$ et $R_k(M/S)^\circ$ le domaine de d\'efinition de $\pi$. La sous-vari\'et\'e $\mathcal{R}_k^c(X)$ de $R_k(M/S)\times R_k(M/S)$ est la cl\^{o}ture de Zariski du produit fibr\'e associ\'e, \textit{i.e.}  $\mathcal{R}_k^c(X) :=\overline{R_k(M/S)^\circ\times_{N} R_k(M/S)^\circ}$. Cette cl\^{o}ture sera aussi not\'ee $R_k(M/S)\times_{N} R_k(M/S)$.
\end{definition}

Dans cette d\'efinition nous avons d\^{u} choisir une int\'egrale premi\`ere maximale du prolongement du champ $X$. Nous donnons deux lemmes qui seront utiles pour montrer que la sous-vari\'et\'e d\'efinie est intrins\`eque au champ $X$ et ne d\'epend pas de l'int\'egrale premi\`ere maximale choisie. Ces lemmes seront aussi utiles dans la preuve de l'\'equivalence des trois d\'efinitions dans la section suivante puis dans le th\'eor\`eme de sp\'ecialisation qui est donn\'e en fin de chapitre.

\begin{lemma} \label{fibredense}
Soit $f:P\rightarrow N$ un morphisme dominant o\`{u} $N$ est une vari\'et\'e irr\'eductible. Soit $O$ un ouvert dense de $P$. Il existe un ouvert $W\subset N$ tel que pour tout $y\in W$, $f^{-1}(y)\cap O$ est un ouvert dense de $f^{-1}(y)$.
\end{lemma}
\begin{proof}
Soient $n$ la dimension de la vari\'et\'e $N$, $(P_i)_{i\in I}$ l'ensemble des composantes irr\'eductibles de $P$, $(P_i)_{i\in J}$ l'ensemble des composantes irr\'eductibles de $P$ qui dominent $N$, $(P_i)_{i\in J'}$ le sous-ensemble des composantes irr\'eductibles de $P$ qui dominent $N$ et telles que $P_i-O$ ne domine pas $N$. Soit $N_0=N-\left(\underset{i\in I-J}{\bigcup}f(P_i)\bigcup \underset{i\in J'}{\bigcup}f(P_i-O)\right)$.

Par d\'efinition, pour tout $i\in J', y\in N_0$, $f\vert_{P_i}^{-1}(y)\cap O=f\vert_{P_i}^{-1}(y)$.
Soit $i\in J-J'$ et $d_i$ la dimension de $P_i$. L'ensemble $P_i-O$ est de dimension inf\'erieure \`a $d_i-1$ car $O$ est un ouvert dense. Une fibre g\'en\'erique de la restriction $f\vert_{P_i}$ est de dimension pure $d_i-n$ et la dimension d'une fibre g\'en\'erique de la restriction $f\vert_{P_i-O}$ est inf\'erieure \`a $d_i-1-n$. Autrement dit, il existe un ouvert $W_i\subset N_0$ tel que pour tout $y\in W_i$, $f\vert_{P_i}^{-1}(y)$ est de dimension pure $d_i-n$ et tel que le compl\'ementaire de $f\vert_{P_i}^{-1}(y)\cap O$ dans $f\vert_{P_i}^{-1}(y)$ est de dimension inf\'erieure \`a $d_i-1-n$. Donc pour tout $y\in W_i$, $f\vert_{P_i}^{-1}(y)\cap O$ est dense dans $f\vert_{P_i}^{-1}(y)$. L'ouvert $W:=\underset{i\in J}{\bigcap}W_i\subset N_0$ est l'ouvert recherch\'e puisque pour tout $y\in W$, $f^{-1}(y)\cap O$ est dense dans $f^{-1}(y)$.
\end{proof}

\begin{lemma}\label{prodfibre}
Soit $f :P\dashrightarrow N$ une application rationnelle dominante o\`{u} $N$ est une vari\'et\'e irr\'eductible. Soit $E$ un ensemble g\'en\'eral de $P$. Le produit fibr\'e $E\times_NE$ est dense dans le produit fibr\'e $P\times_N P$.
\end{lemma}

\begin{proof}
\textit{Fait :} Notons $P^\circ$ le domaine de d\'efinition de $f$. Il existe un ensemble g\'en\'eral $U\subset P$ tel que pour tout $x\in U$, $\{x\}\times f^{-1}(f(x))$ est inclus dans l'adh\'erence de Zariski de $E\times_NE$ dans $P^\circ\times_N P^\circ$. Cette adh\'erence est not\'ee $\overline{E\times_NE}^{P^\circ}$.

\noindent L'ensemble $E$ est l'intersection d\'enombrable d'ouverts $O_i\subset P$, $i\in\mathbb{N}$. Fixons $i\in\mathbb{N}$. Le lemme (\ref{fibredense}) appliqu\'e au morphisme $f:P^\circ\rightarrow N$ et \`a l'ouvert $O_i$ donne l'existence d'un ouvert $W_i\subset N$ tel que pour tout $y\in W_i$, $f^{-1}(y)\cap O_i$ est un ouvert dense de $f^{-1}(y)$. Pour tout $x\in U_i:=f^{-1}(W_i)\cap O_i$, $f^{-1}(f(x))\cap O_i$ est un ouvert dense de $f^{-1}(f(x))$. Notons $U=\underset{i}{\bigcap}U_i$. Pour tout $x\in U$, $f^{-1}(f(x))\cap E$ est dense dans $f^{-1}(f(x))$ car une intersection d\'enombrable d'ouverts denses est dense. Ainsi, pour tout $x\in U\subset E$, $\{x\}\times f^{-1}(f(x))=\overline{\{x\}\times f^{-1}(f(x))\cap E}^{P^\circ}\subset \overline{E\times_NE}^{P^\circ}$ ce qui montre le fait.

Notons maintenant $V=P^\circ\times_NP^\circ-\overline{E\times_NE}^{P^\circ}$ et montrons que cet ouvert de $P^\circ\times_NP^\circ$ est vide. Notons $pr_1:P^\circ\times_NP^\circ\rightarrow P^\circ$ la projection sur la premi\`ere coordonn\'ee.  Si $x\in pr_1(V)$ alors 
 $\{x\}\times f^{-1}(f(x))\not\subset \overline{E\times_NE}^{P^\circ}$. Si $U$ est l'ensemble g\'en\'eral donn\'e dans le fait pr\'ec\'edent alors $pr_1(V)\cap U=\emptyset$. Si $V$ \'etait non vide alors $pr_1(V)$ contiendrait un ouvert de $P^\circ$ ce qui contredirait la vacuit\'e de son intersection avec l'ensemble g\'en\'eral $U$. Donc $E\times_NE$ est dense dans $P^\circ\times_NP^\circ$ qui est lui-m\^{e}me dense dans $P\times_N P$ par d\'efinition. Ceci termine la preuve du lemme.
\end{proof}

\begin{lemma}\label{Rcbiendef}
La sous-vari\'et\'e $\mathcal{R}_k^c(X)$ d\'efinie en (\ref{defRc}) ne d\'epend pas de l'int\'egrale premi\`ere maximale choisie. Elle est de plus invariante sous l'action diagonale du groupe $\Gamma_k$.
\end{lemma}

\begin{proof}
Soient $\pi :R_k(M/S)\dashrightarrow N$ et $\pi' :R_k(M/S)\dashrightarrow N'$ deux int\'egrales premi\`eres maximales. Aux \'egalit\'es $\mathbb{C}(N)=\mathbb{C}(R_k(M/S))^{R_kX}=\mathbb{C}(N')$ il est associ\'e le diagramme commutatif : 
\[\xymatrix{
    R_k(M/S)\ar@{-->>}[d]_{\pi} \ar@{-->>}[dr]^{\pi'}  & \\
   N\ar@{-->}[r]^{\sim}_{\psi}&N'
   }\quad\mbox{qui induit sur des ouverts convenables :}\quad
  \xymatrix{
    R_k(M/S)^\circ\ar@{->>}[d]_{\pi} \ar@{->>}[dr]^{\pi'}  & \\
   N^\circ\ar@{->}[r]^{\sim}_{\psi}&N'^\circ
   } 
\]
La commutativit\'e du dernier diagramme donne l'\'egalit\'e des produits fibr\'es $R_k(M/S)^\circ\times_{N^\circ}R_k(M/S)^\circ=R_k(M/S)^\circ\times_{N'^\circ}R_k(M/S)^\circ$. En appliquant le lemme (\ref{prodfibre}) \`a $\pi$ puis \`a $\pi'$ avec l'ouvert $R_k(M/S)^\circ$ qui est en particulier un ensemble g\'en\'eral, nous obtenons les \'egalit\'es $R_k(M/S)\times_{N} R_k(M/S)=\overline{R_k(M/S)^\circ\times_{N^\circ} R_k(M/S)^\circ}=\overline{R_k(M/S)^\circ\times_{N'^\circ} R_k(M/S)^\circ}=R_k(M/S)\times_{N'} R_k(M/S)$. Ce qui conclut la preuve de la premi\`ere assertion.

Montrons la stabilit\'e sous l'action diagonale. Soit $j_k(\gamma)\in \Gamma_k$. Notons $\Delta_{j_k(\gamma)}$ l'action diagonale de $j_k(\gamma)$ sur le produit $R_k(M/S)\times R_k(M/S)$ et montrons que $\Delta_{j_k(\gamma)}(\mathcal{R}_k^c(X))=\mathcal{R}_k^c(X)$. Soit $\pi:R_k(M/S)\dashrightarrow N$ une int\'egrale premi\`ere maximale du prolongement $R_kX$. Notons $R_k(M/S)^\circ$ le domaine de d\'efinition de $\pi$ et $N'$ une copie de $N$ de telle sorte que l'application $\pi\circ S_{j_k(\gamma^{-1})}$ s'\'ecrive $\pi\circ S_{j_k(\gamma^{-1})}:R_k(M/S)\dashrightarrow N'$. Nous avons l'\'egalit\'e $R_k(M/S)^\circ.{j_k(\gamma)}\times_{N'}R_k(M/S)^\circ.{j_k(\gamma)}=\Delta_{j_k(\gamma)}(R_k(M/S)^\circ\times_N R_k(M/S)^\circ)$. De plus l'application $\pi\circ S_{j_k(\gamma^{-1})}$ est aussi une int\'egrale premi\`ere maximale puisque:
\[
\begin{array}{lll}
(\pi\circ S_{j_k(\gamma^{-1})})^*\mathbb{C}(R_k(M/S))&=&{S_{{j_k(\gamma}^{-1})}}^*\mathbb{C}(R_k(M/S))^{R_kX}\\
&=&\{f\circ S_{j_k(\gamma^{-1})}\ \vert\ R_kX(f)=0, \ f\in\mathbb{C}(R_k(M/S))\}\\
&=&\{h\in\mathbb{C}(R_k(M/S))\ \vert\ R_kX(h\circ S_{j_k(\gamma)})=0 \}\\
&=&\{h\in\mathbb{C}(R_k(M/S))\ \vert\  {S_{j_k(\gamma)}}^*R_kX(h)=0 \}\\
&=& \mathbb{C}(R_k(M/S))^{R_kX}
\qquad\mbox{car ${S_{j_k(\gamma)}}^*R_kX=R_kX$}
\end{array}
\]
La non-d\'ependance envers l'int\'egrale premi\`ere maximale choisie nous permet d'obtenir l'\'egalit\'e $\mathcal{R}_k^c(X)=\overline{R_k(M/S)^\circ\times_{N'}R_k(M/S)^\circ}=\overline{\Delta_{j_k(\gamma)}.R_k(M/S)^\circ\times_{N}R_k(M/S)^\circ}$. Sachant que $\Delta_{j_k(\gamma)}$ est un isomorphisme de $R_k(M/S)\times R_k(M/S)$, il vient $\mathcal{R}_k^c(X)=R_k(M/S)\times_{N'}R_k(M/S)=\Delta_{j_k(\gamma)}(R_k(M/S)\times_{N}R_k(M/S))=\Delta_{j_k(\gamma)}(\mathcal{R}_k^c(X))$ ce qui termine la preuve.
\end{proof}

Nous donnons maintenant le lien entre les vari\'et\'es $\mathcal{R}_k^c(X)$ et $\mathcal{G}_k^c(X)$ qui montre que cette derni\`ere est intrins\`eque \`a la distribution.

\begin{lemma}\label{corresRc}
Les vari\'et\'es $\mathcal{R}_k^c(X)$ et $\mathcal{G}_k^c(X)$ se correspondent par la projection $\Phi_k:R_k(M/S)\times R_k(M/S)\rightarrow Aut_k(M/S)$, \textit{i.e.} $\mathcal{G}_k^c(X)=\Phi_k(\mathcal{R}_k^c(X))$ et $\Phi_k^{-1}(\mathcal{G}_k^c(X))=\mathcal{R}_k^c(X)$. 
\end{lemma}

\begin{proof}
Soient $(H_i)_{1\leq i\leq n}$ une famille g\'en\'eratrice de $\mathbb{C}(R_k(M/S))^{R_kX}$ et $\pi:R_k(M/S)\dashrightarrow N$ une int\'egrale premi\`ere maximale du prolongement $R_kX$. Soit $R_k(M/S)^\circ$ un ouvert de $R_k(M/S)$ sur lequel les applications $H_i$, $\pi$ sont bien d\'efinies. Par d\'efinition, $\mathcal{R}_k^c(X)$ et $\mathcal{G}_k^c(X)$ sont les cl\^{o}tures respectives des ensembles  $R_k(M/S)^\circ\times_NR_k(M/S)^\circ$ et $\{\phi\in \Phi_k(R_k(M/S)^\circ\times R_k(M/S)^\circ)\ \vert\ \forall i,\ H_i\circ\phi=H_i\}$.  Commen\c{c}ons par montrer que ces ensembles se correspondent par $\Phi_k$.
Soient $(j_k(r),j_k(s))\in (R_k(M/S)^\circ)^2$ et $j_k(\phi)=\Phi_k(j_k(r),j_k(s))$. Les \'equivalences suivantes montrent la correspondance: 
\[
\begin{array}{ccl}
&&\forall i=1,\ldots,n,\ H_i\circ\phi=H_i\\
&\Leftrightarrow & \pi(j_k(r)=\pi(j_k(s))\\
&\Leftrightarrow & \forall j_k(\gamma)\in \Gamma_k,\ (j_k(r\circ\gamma),j_k(s\circ\gamma))\in (R_k(M/S)^\circ)^2,\ \pi(j_k(r\circ\gamma))=\pi(j_k(s\circ\gamma)) 
\end{array}
\]
o\`{u} la derni\`ere \'equivalence d\'ecoule du fait que ${S_{j_k(\gamma)}}^*\pi$ est aussi une int\'egrale premi\`ere maximale.
Nous avons l'inclusion $\Phi^{-1}(\mathcal{G}_k^c(X))\supset\mathcal{R}_k^c(X)$ car $\Phi^{-1}(\mathcal{G}_k^c(X))$ est ferm\'e et contient $R_k(M/S)^\circ\times_N R_k(M/S)^\circ$. La stabilit\'e de la sous-vari\'et\'e $\mathcal{R}_k^c(X)$ sous l'action diagonale donne $\Phi_k^{-1}\Phi_k(\mathcal{R}_k^c(X))=\mathcal{R}_k^c(X)$. L'image $\Phi(\mathcal{R}_k^c(X))$ est ferm\'ee puisque la topologie sur $Aut_k(M/S)$ est la topologie quotient. Nous avons l'inclusion $\mathcal{G}_k^c(X)\subset\Phi(\mathcal{R}_k^c(X))$ car $\Phi(\mathcal{R}_k^c(X))$ est un ferm\'e qui contient $\{\phi\in \Phi_k(R_k(M/S)^\circ\times R_k(M/S)^\circ)\ \vert\ \forall i,\ H_i\circ\phi=H_i\}$. Ceci donne l'inclusion $\Phi^{-1}(\mathcal{G}_k^c(X))\subset\mathcal{R}_k^c(X)$ et termine la preuve.
\end{proof}

\subsection{\'Equivalence des trois d\'efinitions}

\begin{theorem}\label{\'egalit\'e}
Soient $S$ et $M$ deux vari\'et\'es complexes, lisses, irr\'eductibles, $\rho:M \rightarrow S$ un morphisme lisse \`a fibres connexes et $X$ un champ de vecteurs sur $M$ tangent aux fibres de $\rho$.
Les trois d\'efinitions de groupo\"{i}de donn\'es dans la sous-section pr\'ec\'edente sont \'equivalentes, \textit{i.e.}  $\mathcal{G}_k^a(X)=\mathcal{G}_k^b(X)=\mathcal{G}_k^c(X)$. Cette sous-vari\'et\'e de $Aut_k(M/S)$ est appel\'e le groupo\"{i}de de Galois d'ordre $k$. Elle est not\'ee $Gal_k(X/S)$.
\end{theorem}

\begin{proof}
Gr\^{a}ce au lemmes (\ref{corresRa}), (\ref{corresRb}) et (\ref{corresRc}), il suffit de montrer les \'egalit\'es $\mathcal{R}_k^a(X)=\mathcal{R}_k^b(X)=\mathcal{R}_k^c(X)$. Nous montrons dans l'ordre: $\mathcal{R}_k^b(X)\subset  \mathcal{R}_k^a(X)$, $\mathcal{R}_k^c(X)\subset \mathcal{R}_k^b(X)$ et $\mathcal{R}_k^a(X)\subset  \mathcal{R}_k^c(X)$.

La vari\'et\'e $\mathcal{R}_k^a(X)$ contient la diagonale de $R_k(M/S)\times R_k(M/S)$ et est tangente au prolongement $0+R_kX$. Elle contient $\mathcal{R}_k^b(X)$ par minimalit\'e de cette derni\`ere.

Soit $\pi:R_k(M/S)\dashrightarrow N$ une int\'egrale premi\`ere maximale de $R_kX$. Le th\'eor\`eme \cite[theorem 1.1.]{Bonnet} nous donne l'existence d'un ensemble g\'en\'eral $E\subset R_k(M/S)$ tel que pour tout $j_k(r)\in E$, $\overline{\pi^{-1}(\pi(j_k(r)))}$ est $V(R_kX,j_k(r))$ la sous-vari\'et\'e minimale tangente \`a $R_kX$ contenant $j_k(r)$. Par d\'efinition, la sous-vari\'et\'e $\mathcal{R}_k^b(X)$ contient la diagonale de $R_k(M/S)\times R_k(M/S)$ et est tangente \`a $0+R_kX$. Pour tout $j_k(r)\in E$, $\{j_k(r)\}\times V(R_kX,j_k(r))\subset \mathcal{R}_k^b(X)$. Donc $E\times_NE\subset \{\{j_k(r)\}\times V(R_kX,j_k(r))\ \vert\ j_k(r)\in E\}\subset \mathcal{R}_k^b(X)$. Gr\^{a}ce au lemme (\ref{prodfibre}) nous obtenons $\mathcal{R}_k^c(X)=\overline{E\times_NE}\subset \mathcal{R}_k^b(X)$.

Il reste la derni\`ere inclusion $\mathcal{R}_k^a(X)\subset \mathcal{R}_k^c(X)$ \`a montrer. Soient $\pi:R_k(M/S)\dashrightarrow N$ une int\'egrale premi\`ere maximale du prolongement $R_kX$ et $R_k(M/S)^\circ$ son domaine de d\'efinition. L'ensemble $R_k(M/S)^\circ\times_{N} R_k(M/S)^\circ$ est tangent au champ $0+R_kX$ puisque le champ $R_kX$ est tangent aux fibres de $\pi$. La propri\'et\'e de tangence est une propri\'et\'e ferm\'ee. La vari\'et\'e $\mathcal{R}_k^c(X)$ est tangente au champ $0+R_kX$. Montrons que cette vari\'et\'e est une relation d'\'equivalence au-dessus d'un ouvert. Ceci nous permettra de conclure sur l'inclusion cherch\'ee gr\^{a}ce \`a la minimalit\'e de la vari\'et\'e $\mathcal{R}_k^a(X)$. Quitte \`a r\'eduire l'ouvert $R_k(M/S)^\circ$, nous pouvons supposer que son image $M^\circ$ par la projection $R_k(M/S)\rightarrow M$ soit un ouvert de $M$.

\textit{Fait: Soit $(j_k(r_1),j_k(r_2),j_k(r_3))\in R^k(M^\circ/S)^3$. Il existe $j_k(\gamma)\in \Gamma_k$ tel que $(j_k(r_1),j_k(r_2),j_k(r_3))\in \left(R_k(M/S)^\circ.j_k(\gamma)\right)^3$}. 

 Pour $i=1,2,3$, les ensembles $\{j_k(\gamma)\in \Gamma_k,\ j_k(r_i).j_k(\gamma)^{-1}\in R_k(M/S)^\circ\}$ sont trois ouverts de $\Gamma_k$. Chacun d'eux est non vide gr\^{a}ce \`a la transitivit\'e de l'action de $\Gamma_k$ sur chaque fibre de $R_k(M^\circ/S)\rightarrow M^\circ$. L'irr\'eductibilit\'e de $\Gamma_k$ implique que leur intersection est non vide. Ce qui montre le fait.

Soit $(j_k(r_1),j_k(r_2),j_k(r_3))\in R_k(M^\circ/S)^3$. Notons $\Delta$ l'action diagonale de $\Gamma_k$. Le fait nous donne l'existence d'un \'el\'ement $j_k(\gamma)\in\Gamma_k$ tel que $(j_k(r_1),j_k(r_2))\in \Delta_{j_k(\gamma)}(R_k(M/S)^\circ\times R_k(M/S)^\circ)$ et $(j_k(r_2),j_k(r_3))\in \Delta_{j_k(\gamma)}(R_k(M/S)^\circ\times R_k(M/S)^\circ)$. Soit $N'$ une copie de $N$ de telle sorte que l'application $\pi\circ S_{j_k(\gamma^{-1})}$ s'\'ecrive $\pi\circ S_{j_k(\gamma^{-1})}:R_k(M/S)\dashrightarrow N'$. Le lemme (\ref{Rcbiendef}) nous dit que $\mathcal{R}_k^c(X)\cap\Delta_{j_k(\gamma)}(R_k(M/S)^\circ\times R_k(M/S)^\circ)=R_k(M/S)^\circ.j_k(\gamma)\times_{N'}R_k(M/S)^\circ.j_k(\gamma)$. Nous obtenons ainsi les implications 
\[
\begin{array}{ccl}
(j_k(r_1),j_k(r_2))\in \mathcal{R}_k^c(X)\vert_{M^\circ\times M^\circ}&\implies &\pi(j_k(r_1\circ\gamma^{-1}))=\pi(j_k(r_2\circ\gamma^{-1}))\\
&\implies & (j_k(r_2),j_k(r_1))\in \mathcal{R}_k^c(X)\vert_{M^\circ\times M^\circ}
\end{array}
\]
et
\[
\begin{array}{ccl}
&(j_k(r_1),j_k(r_2))\in \mathcal{R}_k^c(X)\vert_{M^\circ\times M^\circ}\mbox{ et } (j_k(r_2),j_k(r_3))\in \mathcal{R}_k^c(X)\vert_{M^\circ\times M^\circ}\\
\implies &\pi(j_k(r_1\circ\gamma^{-1}))=\pi(j_k(r_2\circ\gamma^{-1}))\mbox{ et }\pi(j_k(r_2\circ\gamma^{-1}))=\pi(j_k(r_3\circ\gamma^{-1}))\\
\implies &\pi(j_k(r_1\circ\gamma^{-1}))=\pi(j_k(r_3\circ\gamma^{-1}))\\
\implies & (j_k(r_1),j_k(r_3))\in \mathcal{R}_k^c(X)\vert_{M^\circ\times M^\circ}
\end{array}
\]
qui nous donnent la stabilit\'e de la vari\'et\'e $\mathcal{R}_k^c(X)\vert_{M^\circ\times M^\circ}$ par inversion et par composition partielle. Autrement dit, la vari\'et\'e $\mathcal{R}_k^c(X)\vert_{M^\circ\times M^\circ}$ est une relation d'\'equivalence, ce qui conclut la preuve de la derni\`ere inclusion.
\end{proof}

\begin{proposition}\label{dimgroup}
Soit $v\in\mathbb{N}$ le maximum des dimensions des sous-vari\'et\'es minimales tangentes \`a $R_kX$ (voir \cite{Bonnet}). Le groupo\"{i}de de Galois d'ordre $k$ du champ $X$ est une vari\'et\'e irr\'eductible de dimension $\dim(R_k(M/S))+v-\dim(\Gamma_k)$.
\end{proposition} 

\begin{proof}
Nous savons par la d\'efinition dite topologique que le groupo\"{i}de de Galois est la sous-vari\'et\'e minimale tangente \`a $0+R_kX$ qui contient l'identit\'e. Le corollaire \cite[corollary 2.6]{Bonnet} nous dit que cette sous-vari\'et\'e est irr\'eductible. Utilisons la d\'efinition avec les int\'egrales premi\`eres pour le calcul de la dimension. Soit $\pi:R_k(M/S)\dashrightarrow N$ une int\'egrale premi\`ere maximale. Le th\'eor\`eme \cite[theorem 1.1]{Bonnet} nous dit que les fibres g\'en\'eriques sont de dimension $v$ et que la vari\'et\'e $N$ est de dimension $\dim(R_k(M/S))-v$. Le produit fibr\'e $R_k(M/S)\times_NR_k(M/S)$ est de dimension $2\dim(R_k(M/S))-(\dim(R_k(M/S))-v)=\dim(R_k(M/S))+v$. Le groupo\"{i}de de Galois \'etant l'image de ce produit fibr\'e par $\Phi_k$, il est donc de dimension $\dim(R_k(M/S))+v-\dim(\Gamma_k)$.
\end{proof}

\subsection{Le groupo\"ide de Galois}
\begin{lemma}\label{malproj}
La restriction de la projection $Aut_k(M/S)\rightarrow Aut_{k-1}(M/S)$ au groupo\"{i}de de Galois d'ordre $k$ induit un morphisme dominant $Gal_k(X/S)\rightarrow Gal_{k-1}(X/S)$.
\end{lemma}

\begin{proof}
Notons $\pi_k$ la projection $Aut_k(M/S)\rightarrow Aut_{k-1}(M/S)$. Il suffit de montrer que $\overline{\pi_k(Gal_k(X/S))}=Gal_{k-1}(X/S)$. Remarquons d\'ej\`a que $d{\pi_k}.R_k^tX=R_{k-1}^tX$ puisque leurs flots sont les m\^{e}mes, \`a savoir: $\pi_k\circ j_k(\exp(\tau X))=j_{k-1}(\exp(\tau X))$. Nous savons aussi que l'image de l'identit\'e de $Aut_k(M/S)$ par $\pi_k$ est l'identit\'e de $Aut_{k-1}(M/S)$. Nous en d\'eduisons que $\overline{\pi(Gal_k(X/S))}$ est une vari\'et\'e contenant l'identit\'e et tangente \`a $R_{k-1}^tX$ puisque la tangence est une propri\'et\'e ferm\'ee. La minimalit\'e de $Gal_{k-1}(X/S)$ implique l'inclusion  $\overline{\pi(Gal_k(X/S))}\supset Gal_{k-1}(X/S)$. Inversement, la vari\'et\'e $\pi_k^{-1}(Gal_{k-1}(X/S))$ contient l'identit\'e de $Aut_k(M/S)$  et est tangente \`a $R_k^tX$. La minimalit\'e de $Gal_k(X/S)$ implique l'inclusion $Gal_k(X/S)\subset \pi_k^{-1}(Gal_{k-1}(X/S))$. En prenant l'image par $\pi_k$ de part et d'autre de l'inclusion cela donne l'inclusion $\pi_k(Gal_k(X/S))\subset Gal_{k-1}(X/S)$ ce qui termine la preuve.
\end{proof}

Ce lemme permet de donner la d\'efinition suivante:

\begin{definition}\label{maldef}
Le groupo\"ide de Galois $Gal(X/S)$ est la limite projective de la famille des groupo\"{i}des de Galois d'ordre $k$: $(Gal_k(X/S))_k$. Lorsque l'espace des param\`etres $S$ est r\'eduit \`a un point, il est not\'e  $Gal(X)$.
\end{definition}
La d\'efinition du groupo\"{i}de de Galois d'ordre $k$ avec les int\'egrales premi\`eres du champ $X$ am\`ene \`a d\'ecrire le groupo\"ide de Galois de la fa\c{c}on suivante: 
\[
Gal(X/S)=\{\phi\in Aut(M/S)\ \vert\ \forall k\in\mathbb{N},\ \forall H\in\mathbb{C}(R_k(M/S))^{R_kX},\ H\circ j_k(\phi)=H\}
\]

\begin{example}\label{exinvvect}
Soit $Y$ un champ de vecteurs sur $M$ tangent aux fibres de $\rho$. Il d\'efinit $m$ fonctions $H_i\in\mathbb{C}(R_1(M/S))$ donnant les coordonn\'ees de $j_0(r^\ast Y )$:
\[
j_0(r^*Y)=\underset{i}{\sum}H_i(j_1r)\frac{\partial}{\partial \e_i}
\]
Le champ $Y$ est invariant par le champ $X$, \textit{i.e.} $\mathcal{L}_XY=0$, si et seulement si $R_1X(H_i)=0$ pour tout $i$.
 En particulier, si $Y=X$ alors $R_1X(H_i)=0$ pour tout $i$.
 
Pour $\phi\in Aut(M/S)$, par d\'efinition des fonctions $H_i$, $j_0((\phi\circ r)^*Y)=\underset{i}{\sum}H_i(j_1(\phi\circ r))\frac{\partial}{\partial \e_i}$. Ainsi $\phi^*Y=Y$ si et seulement si $H_i \circ R_1(\phi)=H_i$ pour tout $i$. 
\end{example}
\begin{example}\label{exinvforme}
Soit $\omega\in\Omega^1(M/S)$ le dual de l'espace des d\'erivations $\mathcal{O}_S$-lin\'eaire. Cette $1$-forme d\'efinit $m$ fonctions $H_i\in\mathbb{C}(R_1(M/S))$ par 
\[
j_0(r^\ast \omega) = \sum H_i(j_1r) d\epsilon_i
\]
La $1$-forme $\omega$ est invariante par $X$, \textit{i.e.} $\mathcal{L}_X\omega=0$, si et seulement si $R_1X(H_i)=0$ pour tout $i$.
Pour $\phi\in Aut(M/S)$, nous avons l'\'equivalence $\phi^*\omega=\omega$ si et seulement si $H_i\circ R_1\phi=H_i$ pour tout $i$. 
\end{example}

La proposition qui va suivre fait le lien avec entre notre d\'efinition de groupo\"ide de Galois et celle donn\'ee par B. Malgrange (\cite{Malgrange}) dans le cas o\`{u} l'espace des param\`etres $S$ est r\'eduit \`a un point. Commen\c{c}ons par un lemme. L'espace des op\'erateurs diff\'erentiels $\mathcal{D}_{M/S}$ agit sur $\mathcal{O}_{Aut(M/S)}$ (voir la sous-section \ref{structalg}). 

\begin{lemma}\label{lemidealdiff}
Soit $I\subset \mathcal{O}_{Aut(M/S)}$ l'id\'eal associ\'e au groupo\"ide de Galois $Gal(X/S)$. Cet id\'eal est $\mathcal{D}_{M/S}$-invariant. Nous dirons que le groupo\"ide de Galois est une $\mathcal{D}_{M/S}$-sous-pro-vari\'et\'e de $Aut(M/S)$.
\end{lemma}
\begin{proof}
Pour tout $k\in\mathbb{N}$, notons $I_k\subset \mathcal{O}_{R_k(M/S)\times R_k(M/S)}$ l'id\'eal associ\'e \`a la sous-vari\'et\'e de $R_k(M/S)\times R_k(M/S)$ correspondante au groupo\"{i}de de Galois d'ordre $k$: $Gal_k(X/S)$. Par le lemme (\ref{lemdiff}), il suffit de montrer que $(\partial_i\otimes 1+1\otimes \partial_i)I_k\subset I_{k+1}$. Pour $k\in\mathbb{N}$, d\'ecrivons $I_k$ en utilisant la d\'efinition (\ref{Rc}) du groupo\"{i}de de Galois d'ordre $k$. Si $U$ est un ouvert affine de $M$, alors $I_k(U\times U)=\{P\otimes Q-Q\otimes P\ \vert\ (P,Q)\in \mathcal{O}_{R_k(M/S)}(U)^2\mbox{ et }P/Q\in\mathbb{C}(R_k(M/S))^{R_kX}\}$. 
Soit $(P,Q)\in \mathcal{O}_{R_k(M/S)}(U)^2$ v\'erifiant $P/Q\in\mathbb{C}(R_k(M/S))^{R_kX}$. Un petit calcul donne l'\'egalit\'e 
\begin{multline}\tag{$\star$}
(Q\otimes Q)(\partial_i\otimes 1+1\otimes \partial_i)(P\otimes Q-Q\otimes P)+(P\otimes Q-Q\otimes P)(\partial_iQ\otimes Q+Q\otimes \partial_iQ))\\=(Q\partial_iP-P\partial_iQ)\otimes Q^2-Q^2\otimes (Q\partial_iP-P\partial_iQ)
\label{egal}
\end{multline}
Or le fait que $P/Q\in\mathbb{C}(R_k(M/S))^{R_kX}$ et les relation (\ref{relationderivtot}) que v\'erifient le champ de vecteurs $R_kX$ nous donnent $R_{k+1}X\partial_i(P/Q)=\partial_iR_kX(P/Q)=0$. Ceci implique que $\frac{Q\partial_iP-P\partial_iQ}{Q^2}=\partial_i(P/Q)\in \mathbb{C}(R_{k+1}(M))^{R_{k+1}X}$ et donc que $(Q\partial_iP-P\partial_iQ)\otimes Q^2-Q^2\otimes (Q\partial_iP-P\partial_iQ)\in I_{k+1}(U\times U)$. De plus $P\otimes Q-Q\otimes P\in I_k(U\times U)\subset I_{k+1}(U\times U)$ par le lemme (\ref{malproj}). Donc l'\'egalit\'e (\ref{egal}) nous donne $(Q\otimes Q)(\partial_i\otimes 1+1\otimes \partial_i)(P\otimes Q-Q\otimes P)\in I_{k+1}(U\times U)$. L'\'el\'ement $Q\otimes Q$ n'est pas inclus dans $I_{k+1}(U\times U) $ car sinon il devrait s'annuler sur l'identit\'e ce qui impliquerait que $Q=0$. Le groupo\"{i}de de Galois $Gal_{k+1}(X)$ est irr\'eductible donc l'id\'eal $I_{k+1}(U\times U)$ est premier. Ceci entraine que $(\partial_i\otimes 1+1\otimes \partial_i)(P\otimes Q-Q\otimes P)\in I_{k+1}(U\times U)$. Ce qui donne l'inclusion recherch\'ee et montre que le groupo\"ide $Gal(X/S)$ est une $\mathcal{D}_{M/S}$-pro-vari\'et\'e.
\end{proof}

\begin{proposition} \label{prop\'equivdef}
Nous supposons que la vari\'et\'e $S$ est r\'eduite \`a un point. Le groupo\"ide de Galois $Gal(X)$ est la sous-pro-vari\'et\'e minimal de $Aut(M)$ v\'erifiant les trois propri\'et\'es suivantes:
\begin{itemize}
\item $Gal(X)$ est une $\mathcal{D}_{M}$-sous-pro-vari\'et\'e de $Aut(M)$
\item Il existe un ouvert $U\subset M$ au-dessus duquel tous les groupo\"{i}des de Galois d'ordre $k$ sont des sous-groupo\"{i}des alg\'ebriques de $Aut_k(M)$. En suivant (\cite{Malgrange}), si $Gal(X)$ v\'erifie ces deux premi\`eres conditions, alors $Gal(X)$ est appel\'e un groupo\"{i}de de Lie.

\item Les groupo\"{i}des de Galois d'ordre $k$ sont tangents au champ $R_k^tX$, \textit{i.e.}  $R_k^tX\subset TGal_k(X)$
\end{itemize}
\end{proposition} 

La d\'efinition pr\'esente dans \cite{Malgrange} de groupo\"{i}de de Lie  diff\`ere l\'eg\`erement de celle donn\'ee ci-dessus. Pour $n\in\mathbb{N}$, notons $\pi_n$ la projection $Aut(M)\rightarrow Aut_n(M)$. La d\'efinition de \cite{Malgrange} dit qu'une $\mathcal{D}_{M}$-sous-pro-vari\'et\'e $G$ de l'espace des automorphismes de $M$ est un groupo\"{i}de de Lie s'il existe $n\in\mathbb{N}$ tel qu'au-dessus d'un ouvert, $\pi_k(G)$ est un groupo\"{i}de pour $k\geq n$. Mais ces deux d\'efinitions co\"{i}ncident. Les op\'erations de composition partielle et d'inversion commutent aux projections $Aut_n(M)\rightarrow Aut_k(M)$ pour $k\leq n$. Si l'ensemble constructible $\pi_n(G)$ est stable par ces op\'erations au-dessus d'un ouvert alors pour $k\leq n$ l'ensemble constructible $\pi_k(G)$ est aussi stable par ces op\'erations au-dessus d'un ouvert \'eventuellement plus petit. Quitte \`a prendre un plus petit ouvert pour $\pi_n(G)$, nous n'aurons pas besoin de le r\'eduire pour $\pi_k(G)$.

\begin{proof}
Le groupo\"ide de Galois v\'erifie la premi\`ere et la troisi\`eme propri\'et\'e. Nous allons montrer qu'il v\'erifie la deuxi\`eme. Pour montrer que l'ouvert de cette deuxi\`eme condition peut \^{e}tre choisi ind\'ependamment de $k$ nous utilisons un th\'eor\`eme de B. Malgrange qui est donn\'e dans \cite{Malgrange} dans un cadre analytique mais qui reste valable dans un cadre alg\'ebrique. Nous montrons ensuite la minimalit\'e.
 
Pour montrer la deuxi\`eme propri\'et\'e nous allons construire un groupo\"{i}de de Lie puis voir qu'il co\"{i}ncide avec le groupo\"ide de Galois.
Notons $I\subset \mathcal{O}_{Aut(M)}$ l'id\'eal associ\'e \`a $Gal(X)$. Pour $k\in\mathbb{N}$, notons $\pi_k$ la projection $Aut(M)\rightarrow Aut_k(M)$. Pour $k\in\mathbb{N}$, notons $Gal(X)(k)$ la $\mathcal{D}_{M}$-sous-vari\'et\'e de $Aut(M)$ qui est d\'efinie par le faisceau d'id\'eaux diff\'erentiel engendr\'e par $I\cap\mathcal{O}_{Aut_k(M)}$. Montrons que $Gal(X)(k)$ est un groupo\"{i}de de Lie. D'apr\`es le lemme (\ref{lemidealdiff}), $Gal(X)$ est une $\mathcal{D}_{M}$-sous-pro-vari\'et\'e de $Aut(M)$. Donc $Gal(X)\subset Gal(X)(k)$ et $\overline{\pi_k(Gal(X)(k))}=\overline{\pi_k(Gal(X))}=Gal_k(X)$. La d\'efinition (\ref{Rb}) dite alg\'ebrique du groupo\"{i}de de Galois d'ordre $k$ de $X$ nous dit que $\overline{\pi_k(Gal(X)(k))}$ est un groupo\"{i}de au-dessus d'un ouvert.
 Le th\'eor\`eme \cite[4.4.1]{Malgrange} nous dit que $Gal(X)(k)$ est un groupo\"{i}de de Lie.
Le th\'eor\`eme \cite[4.5.1]{Malgrange} nous dit que l'intersection $\underset{k}{\bigcap}Gal(X)(k)$ est encore un groupo\"{i}de de Lie. Pour montrer que le groupo\"ide de Galois $Gal(X)$ est bien un groupo\"{i}de de Lie, il suffit de montrer qu'il co\"{i}ncide avec cette intersection. Or pour tout $k\in\mathbb{N}$, nous avons vu que $Gal(X)\subset Gal(X)(k)$ donc $Gal(X)$ est contenu dans l'intersection. Sachant que pour tout $k\in\mathbb{N}$, $\overline{\pi_k(Gal(X)(k))}=\overline{\pi_k(Gal(X))}$, nous obtenons l'inclusion inverse et terminons la preuve pour la deuxi\`eme propri\'et\'e. 

Il reste la minimalit\'e. Si $W$ est une pro-vari\'et\'e v\'erifiant les trois propri\'et\'es de la proposition alors la minimalit\'e de $W$ implique $W\subset Gal(X)$.  Or $W$ est la limite projective d'une famille de vari\'et\'e $(W_k)_k$. Si l'inclusion est stricte alors il existe un entier $k$ tel que l'inclusion $W_k\subset Gal_k(X)$ est stricte. Mais ceci contredit la minimalit\'e de $Gal_k(X)$. Il vient donc $W=Gal(X)$. Ceci termine la preuve de la proposition.
\end{proof}

\begin{corollary}
Lorsque l'espace des param\`etres $S$ est r\'eduit \`a un point, la d\'efinition (\ref{maldef}) du groupo\"ide de Galois de $X$ co\"{i}ncide avec celle donn\'ee par B. Malgrange dans \cite{Malgrange}.  
\end{corollary}
\begin{proof}
Le groupo\"ide de Galois d\'efini par B. Malgrange est la sous-pro-vari\'et\'e minimal de $Aut(M)$ v\'erifiant les trois propri\'et\'es d\'ecrites dans la proposition (\ref{prop\'equivdef}).  
\end{proof}

Pour conclure cette partie nous allons d\'ecrire une partie des \'equations du groupo\"ide de Galois d'un syst\`eme lin\'eaire. Suivant les cas particuliers effectivement consid\'er\'es d'autres \'equations pourront apparaitre.
\begin{example}
Posons $S =\mathbb C$ et $M = S\times \mathbb C^{n+1}$. Nous noterons $q$ la coordonn\'ee sur $S$ et $t,x_1,\ldots x_n$ les coordonn\'ees sur $\mathbb C^{n+1}$. Consid\'erons un champ de vecteurs de la forme 
\[
X = \frac{\partial}{\partial t} + \sum A_i^j(t,q)x_j\frac{\partial}{\partial x_i}
\]
o\`u  $A\in M_{n\times n}(\mathbb{C}(t,q))$. Ce champ de vecteurs repr\'esente le syst\`eme lin\'eaire \`a param\`etres : $\frac{dX}{dt} = A(t,q)X$.

L'espace des rep\`eres d'ordre 1 est $R_1(M/S) =  M \times GL(\mathbb C^{n+1})$ de coordonn\'ees $q$, $t$, $x_1, \ldots x_n$, $t^{1_i}$et $x_j^{1_i}$, $j=1,\ldots, n$ et $i=0, \ldots, n$. Le prolongement du champ de vecteurs est
\[
R_1X = X + \sum \left(\frac{\partial A_i^j}{\partial t} t^{1_k}x_j + A_i^j x_j^{1_k}\right)\frac{\partial}{\partial x_i^{1_k}}
\]

\begin{enumerate}
\item Les fonctions $t^{1_i}$ sont des int\'egrales premi\`eres de $R_1X$. Elles sont pr\'eserv\'ees par $Gal_1(X/S)$. Ceci signifie que les \'el\'ements de $Gal(X/S)$ sont de la forme :
\[
(M_q,p_1) \to (M_q,p_2) : (t,x) \mapsto (t+c, \varphi_q(t,x))\qquad\mbox{avec }c\in\mathbb{C}  
\]  
\item Pla\c{c}ons nous sur $V \subset R_1(M/S)$ d\'ecrite par $t^{1_i} = \delta_{0i}$ pour $i=0, \ldots, n$. Notons $x'$ la matrice $(x_j^{1_i})_{\substack{j=1,\ldots, n \\ i=1, \ldots, n}}$. Les coefficients de $(x')^{-1}x$ sont des int\'egrales premi\`eres de $R_1X$. Le groupo\"ide les pr\'eservant est l'ensemble des transformations de la forme :
\[
(M_q,p_1) \to (M_q,p_2) : (t,x) \mapsto (t+c, \alpha_q(t)x)   
\]  
avec 
$\frac{d}{dt} \alpha_q = A(t,q)\alpha_q - \alpha_q A (t+c,q)$  

\end{enumerate}

Les r\'eductions suppl\'ementaires qui peuvent \'eventuellement \^etre faites d\'ependront de la matrice $A$. Elles seront donn\'ees par des \'equations polynomiales en les coefficients de $\alpha$,\textit{i.e.} dans $\mathbb C [q,t][\alpha_i^j]$.
  
Notez que le sous-groupo\"ide de $Gal(X/S)$ des \'elements fixant $t$ est le groupe de Galois intrins\`eque sur $\mathbb C(q,t)$ avec sa structure diff\'erentielle par rapport \`a $\frac{d}{dt}$.
\end{example}

\begin{example}
Posons $S =\{\ast\}$ et $M = \mathbb{C}\times \mathbb C^{n+1}$. Nous noterons $q$ la coordonn\'ee sur $\mathbb{C}$ et $t,x_1,\ldots x_n$ les coordonn\'ees sur $\mathbb C^{n+1}$. Consid\'erons un champ de vecteurs de la forme 
\[
X = \frac{\partial}{\partial t} + \sum A_i^j(t,q)x_j\frac{\partial}{\partial x_i}
\]
o\`u  $A\in M_{n\times n}(\mathbb{C}(t,q))$. Ce champ de vecteur repr\'esente le syst\`eme  non lin\'eaire: 
\[
\left\{ \begin{array}{l} \frac{dX}{dt} = A(t,q)X \\ \frac{dq}{dt} = 0\end{array}\right.
\]

L'espace des rep\`eres d'ordre 1 est $R_1(M) =  M \times GL(\mathbb C^{n+2})$ de coordonn\'ees $q$, $t$, $x_1, \ldots x_n$, $q^{1_i}$, $t^{1_i}$et $x_j^{1_i}$, $j=1,\ldots, n$ et $i=-1, \ldots, n$. Le prolongement du champ de vecteurs est
\[
R_1X = X + \sum \left(\left(\frac{\partial A_i^j}{\partial t} t^{1_k}+\frac{\partial A_i^j}{\partial q} q^{1_k}\right)x_j + A_i^j x_j^{1_k}\right)\frac{\partial}{\partial x_i^{1_k}}
\] 

\begin{enumerate}
\item Les fonctions $q$, $q^{1_i}$ et $t^{1_i}$ sont des int\'egrales premi\`eres de $R_1X$. Elles sont pr\'eserv\'ees par $Gal_1(X)$. Ceci signifie que les \'el\'ements de $Gal(X)$ sont de la forme :
\[
(M,p_1) \to (M,p_2) : (q,t,x) \mapsto (q,t+c, \varphi(q,t,x))  
\]  
\item Pla\c{c}ons nous sur $V \subset R_1(M)$ d\'ecrite par $q^{1_{-1}} = 1$, $q^{1_i}=0$ $t^{1_i} = \delta_{0i}$ pour $i=-1, \ldots, n$. Notons $x'$ la matrice $(x_j^{1_i})_{j=1,\ldots, n ; i=1, \ldots, n}$. Les coordonn\'ees de $(x')^{-1}x$ sont des int\'egrales premi\`eres de $R_1X$. Le groupo\"ide les pr\'eservant est l'ensemble des transformation de la forme :
\[
(M,p_1) \to (M,p_2) : (q,t,x) \mapsto (q,t+c, \alpha(q,t)x)   
\]  
avec 
$\frac{\partial}{\partial t} \alpha = A(t,q)\alpha - \alpha A (t+c,q)$  

\end{enumerate}

Les r\'eductions suppl\'ementaires qui peuvent \'eventuellement \^etre faites d\'ependront de la matrice $A$. Elles seront donn\'ees par des \'equations diff\'erentielles en les coefficients de $\alpha$,\textit{i.e.} dans $\mathbb C [q,t][ \frac{\partial^k \alpha_i^j}{\partial q^k}]$.
  
Notez que le sous-groupo\"ide de $Gal(X)$ des \'elements fixant $t$ est une version intrins\`eque du groupe de Galois \`a param\`etres sur $\mathbb C(q)$ avec sa structure diff\'erentielle par rapport \`a $\frac{\partial}{\partial t}$ et $\frac{\partial}{\partial q}$.
\end{example}
\section{Th\'eor\`emes de projection et de sp\'ecialisation}
Ces th\'eor\`emes se proposent de comparer les "tailles" de diff\'erents groupo\"ide de Galois. Avant de les \'enoncer nous devons pr\'eciser la mani\`ere de mesurer la "taille" d'objets de dimension infinie.
\subsection{Type diff\'erentiel du groupo\"ide de Galois}\label{secpoldiff}
Pour mesurer la taille du groupo\"ide de Galois, nous allons regarder la suite des dimensions donn\'ee par la filtration par les jets d'ordre fini. Plus g\'en\'eralement, E.R.Kolchin a donn\'e dans \cite{Kolchin2} un analogue du polyn\^{o}me de Hilbert dans le cadre de l'alg\`ebre diff\'erentielle: 

\begin{theorem}\label{Kolchin}
Soient $(K,D_1,\ldots,D_m)$ un corps diff\'erentiel dans lequel les d\'erivations $D_i$ commutent entre elles, des \'el\'ements $\eta_1,\ldots, \eta_n$ d'une extension diff\'erentielle de $(K,D_1,\ldots,D_m)$, $L=K\langle \eta_1,\ldots \eta_n\rangle$ le corps diff\'erentiel engendr\'e par ces \'el\'ements et pour $k\in\mathbb{N}$, $L^{\leq k}$ le sous corps de $L$ engendr\'e par $K$ et les d\'eriv\'ees au plus $k$-i\`eme de ces \'el\'ements. Il existe un polyn\^{o}me $P\in\mathbb{Q}[x]$ de degr\'e $\ell$ appel\'e polyn\^{o}me de dimension diff\'erentielle tel que:
\begin{itemize}
\item il existe $k_0\in\mathbb{N}$ tel que pour tout $k\geq k_0$, $P(k)=\trdeg_K(L^{\leq k})$
\item $\ell\leq m$ et le polyn\^{o}me $P$ peut s'\'ecrire sous la forme $P(k)=\underset{0\leq i\leq \ell}{\sum}a_i\binom{k+i}{i}$ o\`{u} $a_i\in\mathbb{Z}$
\item le degr\'e $\ell$ et le coefficient $a_\ell$ ne d\'ependent pas des \'el\'ements $\eta_1,\ldots, \eta_n$ mais seulement de $L$
\item l'entier $\ell$ est le type diff\'erentiel de $L$ sur $K$. 
\item le coefficient $a_\ell$ est le degr\'e de transcendance diff\'erentielle de $L$ sur $K$. 
\end{itemize}
\end{theorem} 

\begin{example}\label{extypediff}
Soit $U\subset\mathbb{C}^2$ un ouvert affine muni de deux syst\`emes de coordonn\'ees $(x_1,x_2)$ et $(y_1,y_2)$. Consid\'erons l'extension diff\'erentielle
\[
\left(\mathbb{C},\frac{\partial}{\partial u},\frac{\partial}{\partial v}\right)\subset \left(\mathbb{C}(Aut(U)),\frac{\partial}{\partial x_1},\frac{\partial}{\partial x_2}\right)=\left(\mathbb{C}(x_1,x_2,y_1^\alpha,y_2^\alpha,\alpha\in\mathbb{N}^2),\frac{\partial}{\partial x_1},\frac{\partial}{\partial x_2}\right)
\]
o\`{u} $\frac{\partial}{\partial x_i}x_j=\delta_{ij},\ \frac{\partial}{\partial x_i}y_j^\alpha=y_j^{\alpha+1_i}$.
Le degr\'e de transcendance de $\mathbb{C}(Aut_k(U))$ sur $\mathbb{C}$ est deux fois le cardinal de $\{\alpha\in\mathbb{N}^2\ \vert\ \vert\alpha\vert\leq k\}$ plus deux:
\[
2(1+2+\ldots+(k+1))+2=2\binom{k+2}{2}+2
\]
\end{example}

\begin{corollary}\label{corpoldim}
Il existe un polyn\^{o}me $P\in\mathbb{Q}[x]$ de degr\'e $\ell$ tel que:
\begin{itemize}
\item il existe $k_0\in\mathbb{N}$ tel que pour tout $k\geq k_0$, $P(k)=dim (Gal_k(X/S))$
\item $\ell\leq \dim_S(M)$ et le polyn\^{o}me $P$ peut s'\'ecrire sous la forme $P(k)=\underset{0\leq i\leq l}{\sum}a_i\binom{k+i}{i}$ o\`{u} $a_i\in\mathbb{Z}$
\item le coefficient $a_\ell$ est le degr\'e de transcendance diff\'erentiel de $\mathbb{C}(Gal(X/S))$ sur $\mathbb{C}$
\end{itemize}

\end{corollary}
\begin{proof}
Soit $U\subset M\times M$ le rev\^{e}tement non-ramifi\'e d'un ouvert de $\mathbb{C}^{2(m+d)}$ pr\'esent\'e dans le lemme (\ref{lemautloc}). Ses coordonn\'ees sont not\'ees $x_1,\ldots,x_m,t_1,\ldots,t_d,y_1,\ldots,y_m,z_1,\ldots,z_d$. D'apr\`es ce lemme, $\mathbb{C}(Aut(M/S))=\mathbb{C}(M\times M)(y_i^{\alpha}\mbox{, } 1\leq i\leq m\mbox{, }\alpha\in \mathbb{N}^m\mbox{, }1\leq\vert\alpha\vert)$. 
Les d\'erivations $D_1:=\frac{\partial}{\partial x_1},\ldots,D_m:=\frac{\partial}{\partial x_m}$ engendrent $\mathcal{D}_{M/S}$. Elles v\'erifient $D_ix_j=\delta_{ij},\ D_it_j=0,\ D_iy_j^\alpha=y_j^{\alpha+1_i},\ D_iz_j=0$.
Ainsi, $(\mathbb{C}(Aut(M/S)),D_1,\ldots,D_m)$ est une extension diff\'erentielle de $(\mathbb{C},D_1,\ldots,D_m)$ qui est diff\'erentiellement engendr\'ee par les $2(m+d)$ d'\'el\'ements: $x_1,\ldots,x_m,t_1,\ldots,t_d,y_1,\ldots,y_m,z_1,\ldots,z_d$.

L'id\'eal $I$ associ\'e au groupo\"ide de Galois est premier par la proposition (\ref{dimgroup}) et $\mathcal{D}_{M/S}$-invariant par le lemme (\ref{lemidealdiff}). Les d\'erivations $D_i$ passent au quotient $\mathcal{O}_{Aut(M/S)}(U)/I(U)$. Le corps des fractions du quotient $\mathbb{C}(Gal(X/S))$ est diff\'erentiellement engendr\'e par la classe des $2(m+d)$ d'\'el\'ements $x_1,\ldots,x_m,t_1,\ldots,t_d,y_1,\ldots,y_m,z_1,\ldots,z_d$ modulo $I$. Le th\'eor\`eme (\ref{Kolchin}) nous donne alors l'existence du polyn\^{o}me appel\'e polyn\^{o}me de dimension diff\'erentielle. 
\end{proof}
\begin{definition}
Le type diff\'erentiel du groupo\"ide de Galois du champ $X$ sur $S$ est le degr\'e de son polyn\^{o}me de dimension diff\'erentielle. Il est not\'e $\tdiff(Gal(X/S))$.
\end{definition}
\subsection{Le th\'eor\`eme de projection}\label{secprojection}
Soient $N$ une vari\'et\'e irr\'eductible lisse de dimension $n\leq m$, $\pi :M\dashrightarrow N$ une application rationnelle dominante et $\tilde{\rho}:N\rightarrow S$ un morphisme lisse \`a fibres connexes tels que le diagramme commute
\[
\xymatrix{
    M\ar@{-->>}[r]^{\pi} \ar@{->}[dr]_{\rho}  &N \ar@{->}[d]^{\tilde{\rho}}\\
   &S
   }
\]

\begin{definition}
La sous-vari\'et\'e des $S$-automorphismes d'ordre $k$ de $M$ adapt\'es au feuilletage induit par $\pi$ est not\'ee $Aut_k(\mathcal{F}_{\pi})$, \textit{i.e.}  
\[
Aut_k(\mathcal{F}_{\pi}) :=\overline{\left\{j_k(\phi)\in R(M/S)\ \vert \ j_k(\phi)^*(\pi^*\Omega_N)\subset \pi^*\Omega_N)\right\}}
\]
\end{definition}
Soient $U\subset M\times M$, $\tilde{U}\subset N\times N$ deux voisinages ouverts qui sont respectivement des rev\^{e}tements non ramifi\'es d'ouverts de $\mathbb{C}^{2(m+d)}$, $\mathbb{C}^{2(n+d)}$ et qui font commuter le diagramme:
\[
\xymatrix{
	U\ar[r] \ar[d]_{pr_1/pr_2} & \mathbb{C}^{m+d}\times \mathbb{C}^{m+d}\ar[d]^{pr_1/pr_2} \\
    M\ar[r] \ar[d]_\pi & \mathbb{C}^{m+d}\ar[d] \\
    \tilde{U}\ar[r] \ar[d]_{\tilde{\rho}} & \mathbb{C}^{n+d}\ar[d] \\
    S\ar[r]&\mathbb{C}^d}
\]
Les fl\`eches verticales sans noms du diagramme sont les projections sur les derni\`eres coordonn\'ees. Il vient un syst\`eme de coordonn\'ees $p_1=(p_1',\tilde{p_1},q_1)\in\mathbb{C}^{m-n}\times\mathbb{C}^n\times\mathbb{C}^d,\ p_2=(p_2',\tilde{p_2},q_2)\in\mathbb{C}^{m-n}\times\mathbb{C}^n\times\mathbb{C}^d$ dans lequel un \'el\'ement de $Aut_k(\mathcal{F}_\pi)$ s'\'ecrit:
\begin{multline*}
j_k(\phi)((p_1',\tilde{p_1})+(\e_1,\ldots,\e_m),q_1)=\left(\underset{\substack{\alpha\in \mathbb{N}^m \\ \vert\alpha\vert\leq k}}{\sum}\phi_1^{\alpha}\frac{\epsilon^{\alpha}}{\alpha !},\ldots,\underset{\substack{\alpha\in \mathbb{N}^m \\ \vert\alpha\vert\leq k}}{\sum}\phi_{m-n}^{\alpha}\frac{\epsilon^{\alpha}}{\alpha !},\right.\\ 
\left.\underset{\substack{\alpha\in \{0\}^{m-n}\times\mathbb{N}^n \\ \vert\alpha\vert\leq k}}{\sum}\phi_{m-n+1}^{\alpha}\frac{\epsilon^{\alpha}}{\alpha !},\ldots,\underset{\substack{\alpha\in \{0\}^{m-n}\times\mathbb{N}^n \\ \vert\alpha\vert\leq k}}{\sum}\phi_{m}^{\alpha}\frac{\epsilon^{\alpha}}{\alpha !},q_2\right)
\end{multline*}
o\`{u} pour tout $i=1,\ldots,m,\ \alpha\in\mathbb{N},\ \phi_i^\alpha\in\mathbb{C}$.
Il vient une application $\pi_*:Aut_k(\mathcal{F}_\pi)\rightarrow Aut_k(N/S)$ telle que 
\[
\pi_*(j_k(\phi))(\tilde{p_1}+(\e_{m-n+1},\ldots,\e_m),q_1)=\left(\underset{\substack{\alpha\in \{0\}^{m-n}\times\mathbb{N}^n \\ \vert\alpha\vert\leq k}}{\sum}\phi_{m-n+1}^{\alpha}\frac{\epsilon^{\alpha}}{\alpha !},\ldots,\underset{\substack{\alpha\in \{0\}^{m-n}\times\mathbb{N}^n \\ \vert\alpha\vert\leq k}}{\sum}\phi_{m}^{\alpha}\frac{\epsilon^{\alpha}}{\alpha !},q_2\right)
\]
\begin{theorem}\label{thmproj}
Soit $X$ et $Y$ des champs de vecteurs rationnels sur $M$ et $N$ v\'erifiant $d\pi.X=Y$. Pour tout $k\in\mathbb{N}$, 
\[
\overline{\pi_*Gal_k(X/S))}=Gal_k(Y/S)
\]
\end{theorem} 
\begin{proof}
Notons $exp(\tau X)$ les flots du champ $X$ et $exp(\tau Y)$ les flots du champ $Y$. Nous avons $\pi\circ exp(\tau X)=exp(\tau Y)\circ \pi$. L'action au but des flots $j_k(exp(\tau X))$ sur $Aut_k(M/S)$ pr\'eserve $Aut_k(\mathcal{F}_{\pi})$. Le prolongement $R_k^tX$ est tangent \`a $Aut_k(\mathcal{F}_{\pi})$. Or, l'image de l'identit\'e de $Aut_k(M/S)$ par $\pi_*$ est l'identit\'e de $Aut_k(N/S)$. Par minimalit\'e $Gal_k(X/S)\subset Aut_k(\mathcal{F}_{\pi})$. L'image de $Gal_k(X/S)$ par $\pi_*$ a un sens. 

L'\'egalit\'e $\pi\circ exp(\tau X)=exp(\tau Y)\circ \pi$ donne $\pi_*j_k(exp(\tau X))=j_k(exp(\tau Y))$, donc $\pi_*R_k^tX=R_k^tY$. La sous-vari\'et\'e $\overline{\pi_*(Gal_k(X/S))}\subset Aut(N/S)$ contient l'identit\'e. Elle est tangente \`a $R_k^tY$ puisque la tangence est une propri\'et\'e ferm\'ee. Par minimalit\'e nous avons l'inclusion  $\overline{\pi_*(Gal_k(X/S))}\supset Gal_k(Y/S)$. Inversement, la vari\'et\'e $\pi_*^{-1}(Gal_k(Y/S))$ contient l'identit\'e de $Aut_k(M/S)$  et est tangente \`a $R_k^tX$. La minimalit\'e de $Gal_k(X/S)$ implique l'inclusion $Gal_k(X/S)\subset \pi_*^{-1}(Gal_k(Y/S))$. En prenant l'image par $\pi_*$ de part et d'autre de l'inclusion cela donne l'inclusion $\pi_*(Gal_k(X/S))\subset Gal_k(Y/S)$ ce qui termine la preuve.
\end{proof}
\begin{corollary}
Soit $Y$ un champ de vecteurs rationnel sur $N$ v\'erifiant $d\pi.X=Y$. Alors
\[
\tdiff(Gal(Y/S))\leq \tdiff(Gal(X/S))
\]
\end{corollary}
\begin{proof}
Par le th\'eor\`eme (\ref{thmproj}), pour tout $k\in\mathbb{N}$, $\dim(Gal_k(Y/S))=\dim(\pi_*Gal_k(X/S))$. Donc $\dim(Gal_k(Y/S))\leq \dim(Gal_k(X/S))$. Le polyn\^{o}me de dimension diff\'erentielle du groupo\"ide de Galois de $Y$ sur $S$ croit moins vite que celui de $X$ sur $S$. Ceci donne l'in\'egalit\'e attendue sur les degr\'es de ces polyn\^{o}mes.
\end{proof}

\subsection{Le th\'eor\`eme de sp\'ecialisation}\label{secspe}
Le th\'eor\`eme de sp\'ecialisation compare le groupo\"{i}de de Galois d'ordre $k\in\mathbb{N}$ du champ de vecteurs $X$ tangent aux fibres du morphisme lisse $\rho :M\rightarrow S$ et le groupo\"{i}de de Galois d'ordre $k$ de la restriction $X\vert_q$ de ce champ de vecteurs \`a la fibre $M_q$ de $\rho$ en $q\in S$. 

Remarquons que la fibre $R_k(M/S)_q$ du morphisme $R_k(M/S)\rightarrow S$ en $q\in S$ est exactement l'espace $R_k(M_q)$ des rep\`eres d'ordre $k$ de la fibre $M_q$. Remarquons aussi que la restriction $(R_kX)\vert_q$ du prolongement du champ de vecteurs $X$ \`a la fibre $R_k(M_q)$ et le prolongement $R_kX\vert_q$ de la restriction du champ de vecteurs $X\vert_q$ sur la fibre $R_k(M_q)$ sont deux champs de vecteurs construits avec les m\^{e}mes flots: les flots de la restriction $X\vert_q$ sur $M_q$. Ils sont donc \'egaux. En r\'esum\'e:
\[ 
R_k(M/S)_q=R_k(M_q)\qquad\mbox{et}\qquad (R_kX)\vert_q=R_kX\vert_q
\]

\begin{theorem}\label{thmsp\'ec}
Soient $\rho:M\rightarrow S$ un morphisme lisse \`a fibre connexes entre deux vari\'et\'es lisses irr\'eductibles et $X$ un champ de vecteurs rationnel tangent aux fibres de $\rho$, \textit{i.e.} $d\rho.X=0$. Alors:
\begin{enumerate}
\item Pour tout $q\in S$ o\`{u} cela a un sens : 
\begin{itemize}
\item $Gal_{k}(X\vert_q)\subset Gal_k(X/S)\vert_{M_q \times M_q}$
\item $\dim_S (Gal_k(X/S))\geq \dim_{\mathbb{C}}(Gal_{k}(X\vert_q))$
\end{itemize}
\item Pour $q\in S$ g\'en\'eral, 
 \begin{itemize}
 \item $Gal_{k}(X\vert_q)=Gal_k(X/S)\vert_{M_q\times M_q}$
\item $\dim_S (Gal_k(X/S))= \dim_{\mathbb{C}}(Gal_{k}(X\vert_q))$
\end{itemize}
\end{enumerate}
\end{theorem}

\begin{proof}
Int\'eressons nous d'abord au cas des dimensions.
Suivant P. Bonnet, \cite{Bonnet} nous utiliserons les notations suivantes. Pour $j_k(r)\in R_k(M/S)$ $V(R_kX,j_k(r))$ est la cl\^oture de Zariski de la trajectoire passant par $j_k(r)$. Ce maximum est atteint sur un ensemble g\'en\'eral $E$ de jets de rep\`eres. Pour $q\in S$ o\`{u} cela a un sens, notons $v_q$ le maximum de la dimension de $V(R_kX,j_k(r))=V(R_kX\vert_q,j_k(r))$ pour $j_k(r)\in R_k(M_q)$. Remarquons d\'ej\`a que $v_q\leq v$.
L'image $F$ de $E$ par $R_k(M/S)\rightarrow S$ est un ensemble g\'en\'eral de $S$. Pour tout $q\in F$, l'ensemble $R_k(M_q)\cap E$ est non vide. Ceci donne l'\'egalit\'e $v_q=v$ pour tout $q\in F$. Par la proposition (\ref{dimgroup}), la dimension du groupo\"{i}de de Galois du champ $X$ est $\dim(R_k(M/S))+v-\dim(\Gamma_k)$. De m\^{e}me, $\dim(Gal(R_kX\vert_q))=\dim(R_k(M_q))+v_q-\dim(\Gamma_k)=\dim(R_k(M/S))-\dim(S)+v_q-\dim(\Gamma_k)$. Donc $\dim(Gal(X/S))\geq \dim(Gal(X\vert_q))+\dim(S)$ avec \'egalit\'e pour $q$ dans l'ensemble g\'en\'eral $F$.

Pour les inclusions, nous faisons appel \`a la d\'efinition (\ref{Rb}) dite topologique du groupo\"{i}de de Galois. Pour $q\in S$ o\`{u} cela a un sens, la sous-vari\'et\'e $Gal(X/S)\vert_{R_k(M_q)\times R_k(M_q)}$ est une sous-vari\'et\'e tangente \`a $R_k^tX\vert_q$ et contient l'identit\'e de $Aut(M_q)$. Elle contient donc $Gal(X\vert_q)$ par minimalit\'e de cette derni\`ere. Ceci nous donne les inclusions escompt\'ees.
	
	Enfin, montrons les \'egalit\'es en faisant appel \`a la d\'efinition qui utilise les int\'egrales premi\`eres. L'application quotient $\Phi_k:R_k(M/S)\times R_k(M/S)\rightarrow Aut_k(M/S)$ fait correspondre une sous-vari\'et\'e $\mathcal{R}\subset R_k(M/S)\times R_k(M/S)$ au groupo\"{i}de de Galois d'ordre $k$ de $X$ sur $S$. De m\^{e}me, pour $q\in S$, il correspond une sous-vari\'et\'e $\mathcal{R}_q\subset R_k(M_q)\times R_k(M_q)$ au groupo\"{i}de de Galois d'ordre $k$ de $X\vert_q$. Il suffit de montrer que pour $q\in S$ g\'en\'eral, $\mathcal{R}_q=\mathcal{R}\cap (R_k(M_q)\times R_k(M_q))$. Soit $\pi :R_k(M/S)\dashrightarrow N$ une int\'egrale premi\`ere maximale du champ $X$. Soient $q\in F$, $j_k(r)\in  R_k(M_q)\cap E$. Nous savons que la fibre en $j_k(r)$ de la restriction $\pi\vert_{R_k(M_q)}:R_k(M_q)\dashrightarrow \overline{\pi(R_k(M_q))}$, v\'erifie $\overline{\pi\vert_{R_k(M_q)}^{-1}(\pi\vert_{R_k(M_q)}(j_k(r)))}=V({\mathcal{D}},j_k(r))$. Cette restriction est une int\'egrale premi\`ere maximale du champ $X\vert_q$. Donc $\mathcal{R}_q=R_k(M_q)\times_NR_k(M_q)$. Notons $O$ un ouvert de $N$ inclus dans l'image de $\pi$. Nous savons par le lemme (\ref{prodfibre}) que $\mathcal{R}=\overline{\pi^{-1}(O)\times_O\pi^{-1}(O)}$. Le lemme (\ref{fibredense}) appliqu\'e au morphisme $\mathcal{R}\rightarrow S$, not\'e $\tilde{\rho}$, et \`a l'ouvert dense $\pi^{-1}(O)\times_O\pi^{-1}(O)$ de $\mathcal{R}$ nous dit que pour $q\in S$ g\'en\'erique, $\tilde{\rho}^{-1}(q)\cap \pi^{-1}(O)\times_O\pi^{-1}(O)$ est dense dans $\tilde{\rho}^{-1}(q)$. Autrement dit, pour $q\in S$ g\'en\'erique, $\overline{\pi\vert_{R_k(M_q)}^{-1}(O)\times_O\pi\vert_{R_k(M_q)}^{-1}(O)}=\mathcal{R}\cap R_k(M_q)\times R_k(M_q)$. Au final nous en d\'eduisons que pour $q\in S$ g\'en\'eral, $\mathcal{R}_q=\mathcal{R}\cap R_k(M_q)\times R_k(M_q)$ ce qui ach\`eve la preuve du th\'eor\`eme.
\end{proof}

Voici la cons\'equence directe du r\'esultat de sp\'ecialisation:
 
\begin{theorem}\label{thmcroissance}
Soit $q_0\in S$. Pour $q\in S$ g\'en\'eral,
\[
\tdiff(Gal(X\vert_{q_0}))\leq \tdiff(Gal(X\vert_q))
\]
\end{theorem}
\begin{proof}
Le th\'eor\`eme \ref{thmsp\'ec} nous dit que pour $q\in S$ g\'en\'eral
\[
\dim_{\mathbb{C}}(Gal_{k}(X\vert_q))=\dim_S (Gal_k(X/S))\geq \dim_{\mathbb{C}}(Gal_{k}(X\vert_{q_0}))
\]
 Ceci conclut la preuve du th\'eor\`eme.
\end{proof}

\section{Equations du second ordre}

Dans cette section, nous supposons que $M=\mathbb{C}^3\times S$ dont les coordonn\'ees sont $(x,u,v,q)$. Nous nous int\'eressons \`a l'\'equation
\begin{equation}\tag{$E$}
\frac{d^2u}{dx^2}=F\left(x,u,\frac{du}{dx},q\right)
\label{E}
\end{equation}
o\`{u} $F\in \mathbb{C}(x,u,v,q)$. 
Le groupo\"ide de Galois $Gal(\eqref{E}/\mathbb{C}(S))$ de l'\'equation est le groupo\"ide de Galois $Gal(X_F/S)$ du champ de vecteurs associ\'e: 
\[
X_F=\frac{\partial}{\partial x}+ v\frac{\partial}{\partial u}+F(x,u,v,q)\frac{\partial}{\partial v}
\]
L'\'equation sp\'ecialis\'ee en $q_0\in S$ tel que cela a un sens est 
\begin{equation}\tag{$E(q_0)$}
\frac{d^2u}{dx^2}=F\left(x,u,\frac{du}{dx},q_0\right)
\label{Eq}
\end{equation}
Le groupo\"ide de Galois $Gal(E(q_0))$ de l'\'equation sp\'ecialis\'ee est le groupo\"ide de Galois $Gal(X_F\vert_{q_0})$.
\subsection{Calcul du groupo\"ide de Galois}
Commen\c{c}ons par un lemme qui donne une limite sur le type diff\'erentiel du groupo\"ide de Galois de ces \'equations.

\begin{lemma} \label{lemquad2}
\[
\tdiff(Gal(X_F)/S)\leq 2
\]
\end{lemma}
\begin{proof}
La $1$-forme $dx$ et la fonction $q$ sont pr\'eserv\'ees par le champ $X_F$: $\mathcal{L}_{X_F}dx=0,\ \mathcal{L}_{X_F}q=0$. En se rappelant les exemples (\ref{exinvvect}) et (\ref{exinvforme}), nous avons l'inclusion:
\[
Gal(X_F/S)\subset \left\{ \phi\in Aut(M/S)\ \vert\ \phi^*dx=dx,\ \phi^*q=q,\ \phi^*X_F=X_F\right\}
\]
Soient $U\subset M$ un ouvert analytique, $pr_{uv}$ la projection de $U$ sur les coordonn\'ees $(u,v)$ et $\phi\in Aut(U/S)$. Si l'application $\phi$ pr\'eserve la $1$-forme $dx$ et la fonction $q$ alors elle s'\'ecrit $(x,u,v,q)\rightarrow (x+c,f_1(x,u,v),f_2(x,u,v),q)$. Apr\`es un changement de variables analytiques en diminuant \'eventuellement l'ouvert $U$, le champ de vecteurs $X_F$ se redresse en $\frac{\partial}{\partial x}$. Si $\phi$ pr\'eserve en plus le champ $X_F$ alors elle s'\'ecrit dans les nouvelles coordonn\'ees $(x,u,v,q)\rightarrow (x+c,f_1(u,v),f_2(u,v),q)$. Autrement dit, $Gal(X_F/S)\vert_{U\times U}$ est inclus dans l'ensemble 
\[
\mathcal{G}= \left\{ \phi\in Aut(U/S)\ \big\vert\ \phi(x,u,v,q)=(x+c,f_1(u,v),f_2(u,v),q),\ c\in\mathbb{C},\ \frac{\partial(f_1,f_2)}{\partial(u,v)} \not = 0 \right\}
\]
C'est l'exemple (\ref{extypediff}) donn\'e dans la section (\ref{secpoldiff}). L'ensemble des jets d'ordre $k\in\mathbb{N}$ de $\mathcal{G}$ est une vari\'et\'e analytique de dimension $2+2\binom{k+2}{2}+2+d$.
Pour tout $k\in\mathbb{N}$, la dimension de la vari\'et\'e analytique $Gal_k(X_F)\vert_{U\times U}$ est la m\^{e}me que celle de la vari\'et\'e alg\'ebrique $Gal_k(X_F)$. L'inclusion pr\'ec\'edente nous dit que $\dim(Gal_k(X_F))\leq 2\binom{k+2}{2}+4+d$, ce qui ach\`eve la preuve.
\end{proof}

Soit $U$ un ouvert analytique de $\mathbb{C}^2$. E.Cartan a donn\'e dans (\cite{Cartan}) une classification analytique des sous-groupo\"ides de $Aut(U)$ d\'efinis par des invariants diff\'erentiels et v\'erifiant une hypoth\`ese de r\'egularit\'e appel\'ee involutivit\'e. Cette hypoth\`ese est \'equivalente \`a l'hypoth\`ese d'involutivit\'e pr\'esente dans l'article \cite{Malgrange} de B. Malgrange. Pour chaque param\`etre fix\'e, les th\'eor\`emes ($4.2.2$) et ($4.3.1$) de cet article appliqu\'es \`a notre situation nous disent qu'il existe un ouvert $O\subset \mathbb{C}^3$ Zariski dense au-dessus duquel le groupo\"ide de Galois v\'erifie cette hypoth\`ese.

\begin{theorem}\label{corvolume}
Soient $q_0\in S$, $dvol$ une $3$-forme rationnelle sur $\mathbb{C}^3$ et $O\subset\mathbb{C}^3$ un ouvert sur lequel $dvol$ et $X_F\vert_{q_0}$ sont bien d\'efinis. Notons $Vol(E(q_0))$ la cl\^{o}ture dans $Aut(\mathbb{C}^3)$ de 
\[
\{\phi\in Aut(O)\ \vert\ \phi^*dx=dx,\ \phi^*X_F\vert_{q_0}=X_F\vert_{q_0},\ \phi^*dvol=dvol\}
\]
Si $\mathcal{L}_{X_F\vert_{q_0}}dvol=0$ et $\tdiff(Gal(E(q_0)))=2$, alors $Gal(E(q_0))=Vol(E(q_0))$.
\end{theorem}

\begin{proof}
Puisque $\mathcal{L}_{X_F\vert_{q_0}}dx=0$, $\mathcal{L}_{X_F\vert_{q_0}}X_F\vert_{q_0}=0$ et $\mathcal{L}_{X_F\vert_{q_0}}dvol=0$, nous avons l'inclusion $Gal(E(q_0))\subset Vol(E(q_0))$.

Soit $U\subset O$ un ouvert analytique convenable. La preuve du lemme (\ref{lemquad2}) permet de se ramener \`a
\begin{equation}\tag{$\ast$}
Gal(E(q_0))\vert_{U\times U}\subset  \left\{\phi\in Aut(U)\ \big\vert\ \phi(x,u,v)=(x+c,f_1(u,v),f_2(u,v)),\ \frac{\partial(f_1,f_2)}{\partial(u,v)} \not = 0 \right\}
\label{inclusion}
\end{equation}
 et $dvol=dx\wedge du\wedge dv$. Les fonctions $f_1$ et $f_2$ v\'erifient $jac(f_1,f_2)=1$. Nous devons montrer qu'aucune autre \'equation aux d\'eriv\'ees partielles sur $f_1$ et $f_2$ n'est compatible avec les hypoth\`eses du th\'eor\`eme.

La classification des pseudo groupes de Lie au-dessus d'un ouvert analytique de $\mathbb{C}^2$ aboutit \`a $64$ sous-groupo\"ides dont un r\'ecapitulatif est donn\'e \`a la page $193$ de \cite{Cartan}. Parmi ces sous-groupo\"ides, $31$ sont de dimension fini, \textit{i.e.} de type diff\'erentiel $0$, $26$ d\'ependent de fonctions d'une seule variable, \textit{i.e.} de type diff\'erentiel $1$, $7$ d\'ependent de fonctions de deux variables, \textit{i.e.} de type diff\'erentiel $2$. Ces $7$ derniers sous-groupo\"ides sont d\'enomm\'es $g,g_1, g_2, g_3, g_{01}, g_{22}, g_{29}$. Le sous-groupo\"ide $g$ est le groupo\"ide de toutes les applications. Le sous-groupo\"ide $g_1$ est celui form\'e des applications dont le d\'eterminant jacobien vaut $1$ et aucun des $5$ autres sous-groupo\"ides n'est inclus dans $g_1$. 

Notons $pr_{uv}$ la projection de l'ouvert $U$ sur le plan $\mathbb C^2$ de coordonn\'ees $u$ et $v$. En reprenant l'application ${pr_{uv}}_*:Aut(\mathcal{F}_{pr_{uv}})\rightarrow Aut(pr_{uv}(U))$ de la sous-section (\ref{secprojection}), nous avons par hypoth\`ese $\tdiff({pr_{uv}}_*Gal(E(q_0))\vert_{U\times U})=\tdiff(Gal(E(q_0))\vert_{U\times U})=2$. La classification d\'ecrite ci-dessus donne ${pr_{uv}}_*Gal(E(q_0))\vert_{U\times U}=g_1$. Les fonctions $f_1$ et $f_2$ de l'inclusion \eqref{inclusion} ne v\'erifient pas d'autre \'equation aux d\'eriv\'ees partielles que $\frac{\partial(f_1,f_2)}{\partial (u,v)}=1$. Autrement dit: $Gal(E(q_0))\vert_{U\times U}= Vol(E(q_0))\vert_{U\times U}$.

Notons $Vol_k(E(q_0))$ la projection de $Vol(E(q_0))$ sur l'espace des jets d'ordre $k\in\mathbb{N}$. Par ce qui pr\'ec\`ede, les sous-vari\'et\'es $Gal_k(E(q_0))$ et $Vol_k(E(q_0))$ ont m\^{e}me dimension. Par irr\'eductibilit\'e de cette derni\`ere, l'inclusion $Gal_k(E(q_0))\subset Vol_k(E(q_0))$ donn\'ee en d\'ebut de preuve est une \'egalit\'e : $Gal(E(q_0))=Vol(E(q_0))$.
\end{proof}

\begin{corollary}\label{spvol}
Si le champ de vecteurs $X_F$ associ\'e \`a l'\'equation \eqref{E} pr\'eserve une $3$-forme $fdx\wedge du\wedge dv$ et s'il existe $q_0\in S$ tel que $\tdiff(Gal(X_F\vert_{q_0}))=2$, alors pour $q\in S$ g\'en\'eral, $Gal(E(q))=Vol(E(q))$.
\end{corollary}
\begin{proof}
Par le th\'eor\`eme (\ref{thmcroissance}), pour $q\in S$ g\'en\'eral, $\tdiff(Gal(E(q)))=2$. Par le th\'eor\`eme (\ref{corvolume}), pour ces valeurs des param\`etres, $Gal(E(q))=Vol(E(q))$.
\end{proof}
\subsection{Irr\'eductibilit\'e d'une \'equation}
Nous allons utiliser les r\'esultats pr\'ec\'edents pour obtenir des r\'esultats d'irr\'eductibilit\'e d'une \'equation diff\'erentielle. 

\begin{definition} 
(R\'eductibilit\'e au sens de Nishioka-Umemura)
Soient $q_0\in S$ et $f$ une solution de \eqref{Eq} appartenant \`a une extension diff\'erentielle du corps $(\mathbb{C}(x),\frac{d}{dx})$. Cette solution est r\'eductible si elle appartient \`a une extension diff\'erentielle $K_N$ de $\mathbb{C}(x)$ telle qu'il existe une suite d'extensions $(\mathbb{C}(x),\frac{d}{dx})\subset (K_1,\delta_1)\subset \ldots\subset (K_N,\delta_N)$ v\'erifiant: 
\begin{description}
\item [soit] $K_i$ est alg\'ebrique sur $K_{i-1}$
\item [soit] $K_i=K_{i-1}(g_{pq})$ o\`{u} $(g_{pq})_{pq}$ est une matrice fondamentale de solutions d'un syst\`eme diff\'erentiel lin\'eaire \`a coefficients dans $K_{i-1}$. 
\item [soit] $K_i=K_{i-1}(g)$ o\`{u} $g$ est une solution d'une \'equation diff\'erentielle d'ordre $1$ sur $K_{i-1}$. 
\item [soit] $K_i=K_{i-1}(\varphi(\tau(a_1,\ldots,a_n))\ \vert\ \varphi\in\mathbb{C}(A))$ o\`{u} 
\begin{itemize}
\item $A$ est une vari\'et\'e ab\'elienne
\item $\tau:\mathbb{C}^n\rightarrow A$ est son rev\^etement universel
\item $a_1,\ldots,a_n\in K_{i-1}$
\end{itemize}
\end{description}
\end{definition}

Remarquons que l'extension faisant appel \`a un syst\`eme diff\'erentiel lin\'eaire n'est pas suppos\'ee de Picard-Vessiot, il pourra \^{e}tre ajout\'ees de nouvelles constantes. Remarquons aussi que dans l'extension du troisi\`eme type, nous nous autorisons \`a r\'esoudre n'importe quelle \'equation diff\'erentielle ordinaire non-lin\'eaire d'ordre $1$.

Une \'equation peut avoir une solution r\'eductible sans que nous n'ayons aucune information sur les autres solutions:
\begin{example}
Pour $(P,Q)\in \mathbb{C}[x,u,v]^2$, l'\'equation $(E):u''=P(x,u,u')u+Q(x,u,u')u'$ admet $0$ pour solution r\'eductible.
\end{example}

Il nous faut alors une hypoth\`ese de g\'en\'ericit\'e sur les solutions r\'eductibles :
\begin{definition}
Soit $q_0\in S$. Une solution $f$ de \eqref{Eq} est appel\'ee solution g\'en\'erale si 
\[
\trdeg_{\mathbb{C}(x)}(\mathbb{C}(x,f,f'))=2
\]
 Si l'\'equation \eqref{Eq} poss\`ede une solution r\'eductible g\'en\'erale, alors \eqref{Eq} est dite r\'eductible. Sinon elle est dite irr\'eductible.
\end{definition}
\begin{example}
L'\'equation $(E) :u''=0$ est r\'eductible puisqu'elle admet $f=ax+b$ comme solution r\'eductible g\'en\'erale o\`{u} $a$ et $b$ sont des constantes transcendantes sur $\mathbb{C}$. Remarquons que l'extension $\mathbb{C}(x)\subset \mathbb{C}(x,a,b)$ n'est pas de Picard Vessiot.
\end{example}
Le lien entre la r\'eductibilit\'e d'une \'equation et la structure du groupo\"ide de Galois du champ de vecteurs associ\'e est donn\'e par G.Casale dans \cite{CasaleP1}.

\begin{theorem}\label{croissancered}
Soit $q_0\in S$. Si $Gal(E(q_0))=Vol(E(q_0)$ alors \eqref{Eq} est irr\'eductible.
\end{theorem}

\begin{corollary}\label{speirred}
Si le champ de vecteurs $X_F$ associ\'e \`a l'\'equation \eqref{E} pr\'eserve une $3$-forme et s'il existe $q_0\in S$ tel que $\tdiff(Gal(E(q_0)))=2$ alors pour $q\in S$ g\'en\'eral, l'\'equation $(E_q)$ est irr\'eductible.
\end{corollary}
\begin{proof}
Par le corollaire \ref{spvol}, pour $q\in S$ g\'en\'eral, $Gal(E(q))=Vol(E(q))$. Par le th\'eor\`eme \ref{croissancered}, pour ces valeurs des param\`etres, $(E_q)$ est irr\'eductible.
\end{proof}
\begin{remark}
Dans le corollaire (\ref{speirred}), l'hypoth\`ese sur le type diff\'erentiel du groupo\"ide de Galois de $(E(q_0))$ ne peut pas \^{e}tre remplac\'ee par l'hypoth\`ese d'irr\'eductibilit\'e de $(E(q_0))$. En effet, l'\'equation $P_{{\textsc{\romannumeral 6}}}(0,0,0,1/2)$ donn\'ee ci-dessous est irr\'eductible mais le type diff\'erentiel de son groupo\"ide de Galois est $0$ (voir \cite{Casale6, Watanabe6}.
\end{remark}





\subsection{Les \'equations de Painlev\'e}\label{secpainlev\'e}
La classification des solutions r\'eductibles des \'equations de Painlev\'e est connue (voir \cite{Lisovyy, Murata, Nishioka, Noumi, Umemura1, Umemura24, Umemura3, Watanabe5, Watanabe6}). Les th\'eor\`emes pr\'ec\'edents vont nous permettre de montrer l'irr\'eductibilit\'e des \'equations pour des valeurs g\'en\'erales des param\`etres. Le r\'esultat que nous retrouvons est un peu plus faible que le r\'esultat connu mais la preuve n'est pas sp\'ecialis\'ee aux \'equations de Painlev\'e.
Voici la liste de ces \'equations qui peut \^{e}tre trouv\'ee dans \cite{Ohyama}:
\begin{align}
u''&=6u^2+x \tag{$P_{{\textsc{\romannumeral 1}}}$}\\
u''&=2u^3+xu+\alpha \tag{$P_{{\textsc{\romannumeral 2}}}(\alpha)$}\\
u''&=\frac{u'^2}{u}-\frac{u'}{x}+\frac{\alpha u^2+\beta}{x}+\gamma u^3+\frac{\delta}{u} \tag{$P_{{\textsc{\romannumeral 3}}}(\alpha,\beta,\gamma,\delta)$}\\
u''&=\frac{u'^2}{2u}+\frac{3}{2}u^3+4xu^2+2(t^2-\alpha)u+\frac{\beta}{u} \tag{$P_{{\textsc{\romannumeral 4}}}(\alpha,\beta)$}\\
u''&=\left(\frac{1}{2u}+\frac{1}{u-1}\right)u'^2-\frac{u'}{x}+\frac{(u-1)^2}{x^2}\left(\alpha u+\frac{\beta}{u}\right)+\gamma \frac{u}{x}+\delta\frac{u(u+1)}{u-1} \tag{$P_{{\textsc{\romannumeral 5}}}(\alpha,\beta,\gamma,\delta)$}\\
u''&=\frac{1}{2}\left(\frac{1}{u}+\frac{1}{u-1}+\frac{1}{u-x}\right)u'^2-\left(\frac{1}{x}+\frac{1}{x-1}+\frac{1}{u-x}\right)u'
\nonumber\\&\hfill +\frac{u(u-1)(u-x)}{x^2(x-1)^2}\left(\alpha+\beta\frac{x}{u^2}+\gamma\frac{x-1}{(u-1)^2}+\delta\frac{x(x-1)}{(u-x)^2}\right) \tag{$P_{{\textsc{\romannumeral 6}}}(\alpha,\beta,\gamma,\delta)$}
\end{align}
Pour $J=\textsc{\romannumeral 1},\ldots,\textsc{\romannumeral 6}$, l'espace des param\`etres $S_J$ est $\mathbb{C}^d$ o\`{u} $d=0,1,2,4$ d\'epend de l'\'equation consid\'er\'ee.

G. Casale dans (\cite{CasaleP1}), G. Casale et J. A. Weil dans (\cite{CasaleWeil}), F. Loray et S. Cantat dans (\cite{Cantat}) ont respectivement calcul\'e le groupo\"ide de Galois des \'equation $P_{{\textsc{\romannumeral 1}}},\ P_{{\textsc{\romannumeral 2}}}(0)$ et $P_{{\textsc{\romannumeral 6}}}(\alpha,\beta,\gamma,\delta)$: 
\begin{theorem}\label{casalespe}
Nous avons les \'egalit\'es de groupo\"ides suivantes:
\begin{itemize}
\item $Gal(P_{{\textsc{\romannumeral 1}}})=Vol(P_{{\textsc{\romannumeral 1}}})$ 
\item $Gal(P_{{\textsc{\romannumeral 2}}}(0))=Vol(P_{{\textsc{\romannumeral 2}}}(0))$ 
\item pour $(\alpha,\beta,\gamma,\delta)\in\mathbb{C}^4$ g\'en\'erique, $Gal(P_{{\textsc{\romannumeral 6}}}(\alpha,\beta,\gamma,\delta))=Vol(P_{{\textsc{\romannumeral 6}}}(\alpha,\beta,\gamma,\delta))$
\end{itemize}
\end{theorem}

\begin{corollary}\label{2emepainleve}
Pour $\alpha\in\mathbb{C}$ g\'en\'eral, $Gal(P_{{\textsc{\romannumeral 2}}}(\alpha))=Vol(P_{{\textsc{\romannumeral 2}}}(\alpha))$ et $P_{{\textsc{\romannumeral 2}}}(\alpha)$ est irr\'eductible.
\end{corollary}
\begin{proof}
Par le th\'eor\`eme (\ref{casalespe}), $\tdiff(Gal(P_{{\textsc{\romannumeral 2}}}(0)))=2$. Le champ de vecteurs associ\'e \`a l'\'equation $P_{{\textsc{\romannumeral 2}}}$ pr\'eserve la $3$-forme $dx\wedge du\wedge dv$. Par le corollaire (\ref{spvol}), pour $\alpha\in\mathbb{C}$ g\'en\'eral,  $Gal(P_{{\textsc{\romannumeral 2}}}(\alpha))=Vol(P_{{\textsc{\romannumeral 2}}}(\alpha))$. Par le th\'eor\`eme (\ref{croissancered}), pour ces valeurs du param\`etre, $P_{{\textsc{\romannumeral 2}}}(\alpha)$ est irr\'eductible.
\end{proof}

Pendant l'\'etude des \'equations de Painlev\'e, il a \'et\'e d\'ecouvert qu'elles d\'eg\'en\'eraient les unes sur les autres en suivant le diagramme suivant (voir \cite{Ohyama}):
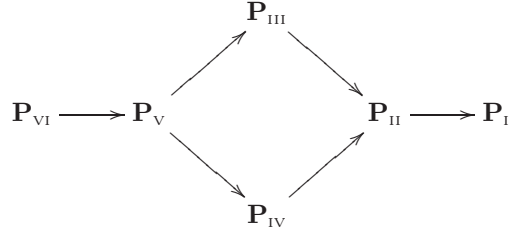
\[
\xymatrix{
&&\mathbf{P_{{\textsc{\romannumeral 3}}}}\ar[dr]\\
\mathbf{P_{{\textsc{\romannumeral 6}}}}\ar@{->}[r] &\mathbf{P_{{\textsc{\romannumeral 5}}}}\ar@{->}[ur]\ar[dr]&&\mathbf{P_{{\textsc{\romannumeral 2}}}}\ar@{->}[r]&\mathbf{P_{{\textsc{\romannumeral 1}}}} \\
&&\mathbf{P_{{\textsc{\romannumeral 4}}}}\ar@{->}[ur]}
\]

Nous allons utiliser ces d\'eg\'en\'erescences pour obtenir les m\^{e}mes r\'esultats sur les groupo\"ides de Galois et sur l'irr\'eductibilit\'e des autres \'equations de Painlev\'e. 

\begin{proposition}\label{thmpainlev\'e}
Pour $J={\textsc{\romannumeral 1}},{\textsc{\romannumeral 2}},{\textsc{\romannumeral 3}},{\textsc{\romannumeral 4}},{\textsc{\romannumeral 5}},{\textsc{\romannumeral 6}}$ et $q\in S_J$ g\'en\'eral,
\[
\tdiff(Gal(P_{J}(q)))=2
\] 
\end{proposition} 
\begin{proof}
Pour $J={\textsc{\romannumeral 1}},{\textsc{\romannumeral 2}}$, c'est le th\'eor\`eme (\ref{casalespe}) et le corollaire (\ref{2emepainleve}).
Nous d\'etaillons la preuve pour l'\'equation $P_{{\textsc{\romannumeral 3}}}$. Pour les autres \'equations, la m\'ethode reste la m\^{e}me. Les changements de variables qui permettent de faire d\'eg\'en\'erer les \'equations de Painlev\'e les unes sur les autres sont donn\'es dans \cite{Ohyama}. Le changement de variables qui permet de faire d\'eg\'en\'erer $P_{{\textsc{\romannumeral 3}}}$ sur $P_{{\textsc{\romannumeral 2}}}$ est:
\[
\begin{array}{cccc}
	\phi_{}:&\mathbb{C}^3\times \mathbb{C}\times\mathbb{C}^*&		\longrightarrow & \mathbb{C}^3\times\mathbb{C}^4\\
	&\left( t,f,g,a,\e\right)
	&\longmapsto&
	\displaystyle{\left(1+\e^2 t,1+2\e f,\frac{2g}{\e},\tilde{\phi}(\e,a)\right)}
\end{array}
\] o\`{u} 
\[
\tilde{\phi}(a,\e)= \left(
 -\frac{1}{2\e^{6}},
 \frac{2a}{\e^{3}}+\frac{1}{2\e^{6}},
 \frac{1}{4\e^{6}},
 -\frac{1}{4\e^{6}}\right)
\]
 Notons $X_{{\textsc{\romannumeral 2}}}$ le champ de vecteurs rationnel sur $\mathbb{C}^3\times\mathbb{C}\times\{0\}$ associ\'e \`a $P_{{\textsc{\romannumeral 2}}}$ et $X_{{\textsc{\romannumeral 3}}}$ le champ de vecteurs rationnel sur $\mathbb{C}^3\times \mathbb{C}^4$ associ\'e \`a $P_{{\textsc{\romannumeral 3}}}$. Autrement dit $$X_{{\textsc{\romannumeral 2}}}=\displaystyle{\frac{\partial}{\partial t}+ g\frac{\partial}{\partial f}+F_{{\textsc{\romannumeral 2}}}(t,f,g,a)\frac{\partial}{\partial g}}\quad \mbox{et}\quad X_{{\textsc{\romannumeral 3}}}=\displaystyle{\frac{\partial}{\partial x}+ v\frac{\partial}{\partial u}+F_{{\textsc{\romannumeral 3}}}(x,u,v,\alpha,\beta,\gamma,\delta)\frac{\partial}{\partial v}}$$ o\`{u} 
\begin{itemize}
\item $F_{{\textsc{\romannumeral 2}}}(t,f,g,a)= 2f^3+tf+a$
\item $\displaystyle{F_{{\textsc{\romannumeral 3}}}(x,u,v,\alpha,\beta,\gamma,\delta)=\frac{v^2}{u}-\frac{v}{x}+\frac{\alpha u^2+\beta}{x}+\gamma u^3+\frac{\delta}{u}}$
\end{itemize}
Le feuilletage donn\'e par les trajectoires de $X_{{\textsc{\romannumeral 3}}}$ est d\'ecrit par le syst\`eme de formes \[d\alpha,\ d\beta,\ d\gamma,\ d\delta,\ du-vdx,\ dv-F_{{\textsc{\romannumeral 3}}}dx\] Le tir\'e en arri\`ere par $\phi$ sur  $\mathbb{C}^3\times\mathbb{C}^*\times\mathbb{C}$ est d\'ecrit par le syst\`eme de formes \[d\e,\ da,\ df-gdt,\ dg-\frac{\e^3}{2}F_{{\textsc{\romannumeral 3}}}\circ\phi dt\]
\textit{i.e.} c'est le feuilletage donn\'e par le champ $Y_{{\textsc{\romannumeral 3}}}:=\e^2\phi^*X_{{\textsc{\romannumeral 3}}}$ qui correspond \`a l'\'equation du second ordre : $f''=\frac{\e^3}{2}F_{{\textsc{\romannumeral 3}}}(\phi(t,f,g,\e,a))$. La d\'eg\'en\'erescence de $P_{{\textsc{\romannumeral 3}}}$ sur $P_{{\textsc{\romannumeral 2}}}$ s'exprime de la mani\`ere suivante : $\frac{\e^3}{2} F_{{\textsc{\romannumeral 3}}}\circ\phi\vert_{\e=0}=F_{{\textsc{\romannumeral 2}}}$, autrement dit  $Y_{{\textsc{\romannumeral 3}}}\vert_{\e=0}=X_{{\textsc{\romannumeral 2}}}$.
La situation de d\'eg\'en\'erescence \'etant d\'ecrite nous pouvons appeler les th\'eor\`emes rencontr\'es dans la section pr\'ec\'edente.
%

\noindent Par le th\'eor\`eme (\ref{casalespe}), $\tdiff(Gal(Y_{{\textsc{\romannumeral 3}}}\vert_{(0,0)}))=2$. Le th\'eor\`eme (\ref{thmcroissance}) nous dit que pour $(a,\e)$ g\'en\'eral dans l'espace des param\`etres $\mathbb{C}^2$, $\tdiff(Gal(Y_{{\textsc{\romannumeral 3}}}\vert_{(a,\e)}))=2$. Choisissons $(a_0,\e_0)\in\mathbb{C}\times \mathbb{C}^*$ l'un de ces param\`etres pour lequel la deuxi\`eme coordonn\'ee $\e_0$ est non nulle. La restriction $\phi\vert_{(a_0,\e_0)}$ conjugue $Y_{{\textsc{\romannumeral 3}}}\vert_{(a_0\e_0)}$ \`a $\e_0^2X_{{\textsc{\romannumeral 3}}}\vert_{\tilde{\phi}(a_0,\e_0)}$. En appliquant le th\'eor\`eme de projection (\ref{thmproj}) \`a $\phi\vert_{(a_0,\e_0)}$ et \`a $\phi\vert_{(a_0,\e_0)}^{-1}$, il vient que le groupo\"ide de Galois du champ $Y_{{\textsc{\romannumeral 3}}}\vert_{(a_0,\e_0)}$ est isomorphe au groupo\"ide de Galois du champ $X_{{\textsc{\romannumeral 3}}}\vert_{\tilde{\phi}(a_0,\e_0)}$.
 Le th\'eor\`eme (\ref{thmcroissance}) nous dit que pour $(\alpha,\beta,\gamma,\delta)$ g\'en\'eral dans l'espace des param\`etres $\mathbb{C}^4$, $\tdiff(Gal(X_{{\textsc{\romannumeral 3}}}\vert_{(\alpha,\beta,\gamma,\delta)}))=2$. Ceci qui conclut la preuve pour $P_{{\textsc{\romannumeral 3}}}$.


Le changement de variables qui permet de faire d\'eg\'en\'erer $P_{{\textsc{\romannumeral 4}}}$ sur $P_{{\textsc{\romannumeral 2}}}$ est
\[
\begin{array}{cccc}
	\phi_{}:&\mathbb{C}^3\times \mathbb{C}\times\mathbb{C}^*&	\longrightarrow & \mathbb{C}^3\times\mathbb{C}^2\\
	&\left( t,f,g,a,\e\right)
	&\longmapsto&
	\displaystyle{\left(\frac{\e t}{2^{2/3}}-\frac{1}{\e^3},\frac{2^{2/3} f}{\e}+\frac{1}{\e^3},\frac{2^{4/3}g}{\e^2},\tilde{\phi}(a,\e)\right)}
\end{array}
\] o\`{u} 
\[
\tilde{\phi}(\e,a)= \left(
 -2a-\frac{1}{2\e^{6}} , -\frac{1}{2\e^{12}}  \right)
\]
Soient $X_{{\textsc{\romannumeral 4}}}$ le champs de vecteurs sur $\mathbb{C}^3\times \mathbb{C}^2$ associ\'e \`a $P_{{\textsc{\romannumeral 4}}}$ et $Y_{{\textsc{\romannumeral 4}}}:=\displaystyle{\frac{\e}{2^{2/3}}}\phi^{*}X_{{\textsc{\romannumeral 4}}}$ champ de vecteurs sur $\mathbb{C}^3\times\mathbb{C}\times\mathbb{C}$. La d\'eg\'en\'erescence de $P_{{\textsc{\romannumeral 4}}}$ sur $P_{{\textsc{\romannumeral 2}}}$ s'exprime de la mani\`ere suivante : $Y_{{\textsc{\romannumeral 4}}}\vert_{\e=0}=X_{{\textsc{\romannumeral 2}}}$ o\`{u} $X_{{\textsc{\romannumeral 2}}}$ est le champ de vecteurs donn\'e pr\'ec\'edemment. Nous avons $\tdiff(Gal(Y_{{\textsc{\romannumeral 4}}}\vert_{(0,0)}))=2$. En reprenant le m\^{e}me raisonnement il s'ensuit que pour des param\`etres $(\alpha,\beta)\in\mathbb{C}^2$ g\'en\'eraux, $\tdiff(Gal(X_{{\textsc{\romannumeral 4}}}\vert_{(\alpha,\beta)}))=2$.

Le changement de variables qui permet de faire d\'eg\'en\'erer $P_{{\textsc{\romannumeral 5}}}$ sur $P_{{\textsc{\romannumeral 4}}}$ est
\[
\begin{array}{cccc}
	\phi:&\mathbb{C}^3\times\mathbb{C}^2\times \mathbb{C}^*&\longrightarrow & \mathbb{C}^3\times\mathbb{C}^4\\
	&\left( t,f,g,a,b,\e\right)
	&\longmapsto&
	\displaystyle{\left(1+\e\sqrt{2}t,\frac{\e f}{\sqrt{2}},\frac{g}{2},\tilde{\phi}(a,b,\e)\right)}
\end{array}
\] o\`{u} 
\[
\tilde{\phi}(a,b,\e)= \left(
 \frac{1}{2\e^{4}},
 \frac{b}{4} ,
-\frac{1}{\e^4},
 \frac{a}{\e^2}-  \frac{1}{2\e^{4}}\right)
\]
Soient $X_{{\textsc{\romannumeral 5}}}$ le champ de vecteurs sur $\mathbb{C}^3\times\mathbb{C}^4$ associ\'e \`a $P_{{\textsc{\romannumeral 5}}}$ et $Y_{{\textsc{\romannumeral 5}}}:=\e\sqrt{2}\phi^{*}X_{{\textsc{\romannumeral 5}}}$ champ de vecteurs sur $\mathbb{C}^3\times\mathbb{C}^2\times\mathbb{C}$. La d\'eg\'en\'erescence de $P_{{\textsc{\romannumeral 5}}}$ sur $P_{{\textsc{\romannumeral 4}}}$ s'exprime de la mani\`ere suivante :  $Y_{{\textsc{\romannumeral 5}}}\vert_{\e=0}=X_{{\textsc{\romannumeral 4}}}$ o\`{u} le champ $X_{{\textsc{\romannumeral 4}}}$ est cette fois-ci d\'efini sur $\mathbb{C}^3\times\mathbb{C}^2\times\{0\}$. Nous savons maintenant que pour des param\`etres $(a,b)\in\mathbb{C}^2$ g\'en\'eraux, $\tdiff(Gal(Y_{{\textsc{\romannumeral 5}}}\vert_{(a,b,0)}))=2$. Choisissons en un. En reprenant le m\^{e}me raisonnement avec ce param\`etre il s'ensuit que pour des param\`etres $(\alpha,\beta,\gamma,\delta)\in\mathbb{C}^4$ g\'en\'eraux, $\tdiff(Gal(X_{{\textsc{\romannumeral 5}}}\vert_{(\alpha,\beta,\gamma,\delta)}))=2$.

Enfin, le changement de variables qui permet de faire d\'eg\'en\'erer $P_{{\textsc{\romannumeral 6}}}$ sur $P_{{\textsc{\romannumeral 5}}}$ est
\[
\begin{array}{cccc}
	\phi_{}:&\mathbb{C}^3\times\mathbb{C}^4\times \mathbb{C}^*&	\longrightarrow & \mathbb{C}^3\times\mathbb{C}^4\\
	&\left( t,f,g,a,b,c,d,\e\right)
	&\longmapsto&
	\displaystyle{\left(1+\e t, f,\frac{g}{\e},\tilde{\phi}(a,b,c,d,\e)\right)}
\end{array}
\] o\`{u} 
\[
\tilde{\phi}(a,b,c,d,\e)= \left(
a,
b ,
-\frac{d}{\e^2}+\frac{c}{\e},
 \frac{d}{\e^2} \right)
\]
Soient $X_{{\textsc{\romannumeral 6}}}$ le champs de vecteurs sur $\mathbb{C}^3\times\mathbb{C}^4$ associ\'e \`a $P_{{\textsc{\romannumeral 6}}}$ et $Y_{{\textsc{\romannumeral 6}}}:=\e \phi^{*}X_{{\textsc{\romannumeral 6}}}$ champ de vecteurs sur $\mathbb{C}^3\times\mathbb{C}^4\times\mathbb{C}$. La d\'eg\'en\'erescence de $P_{{\textsc{\romannumeral 6}}}$ sur $P_{{\textsc{\romannumeral 5}}}$ s'exprime de la mani\`ere suivante : $Y_{{\textsc{\romannumeral 6}}}\vert_{\e=0}=X_{{\textsc{\romannumeral 5}}}$ o\`{u} cette fois-ci le champ $X_{{\textsc{\romannumeral 5}}}$ est d\'efini sur $\mathbb{C}^3\times\mathbb{C}^4\times\{0\}$. Nous savons maintenant que pour des param\`etres $(a,b,c,d)\in\mathbb{C}^4$ g\'en\'eraux, $\tdiff(Gal(Y_{{\textsc{\romannumeral 5}}}\vert_{(a,b,c,d,0)}))=2$. Choisissons en un. En reprenant le m\^{e}me raisonnement il s'ensuit que pour des param\`etres $(\alpha,\beta,\gamma,\delta)\in\mathbb{C}^4$ g\'en\'eraux, $\tdiff(Gal(X_{{\textsc{\romannumeral 6}}}\vert_{(\alpha,\beta,\gamma,\delta)}))=2$.
\end{proof}

\begin{theorem}\label{corirred}
Pour $J={\textsc{\romannumeral 1}},{\textsc{\romannumeral 2}},{\textsc{\romannumeral 3}},{\textsc{\romannumeral 4}},{\textsc{\romannumeral 5}},{\textsc{\romannumeral 6}}$ et pour $q\in S_J$ g\'en\'eral, $Gal(P_J(q))=Vol(P_J(q))$.
\end{theorem}
\begin{proof}
Pour $J={\textsc{\romannumeral 1}},{\textsc{\romannumeral 2}}$, c'est le th\'eor\`eme (\ref{casalespe}) et le corollaire (\ref{2emepainleve}). 

Par la proposition (\ref{thmpainlev\'e}), pour $J=,{\textsc{\romannumeral 3}},{\textsc{\romannumeral 4}},{\textsc{\romannumeral 5}},{\textsc{\romannumeral 6}}$ et $q\in S_J$ g\'en\'eral, $\tdiff(Gal(P_{J}(q)))=2$. Fixons $J$ et $q$ parmi ces valeurs. Par le th\'eor\`eme (\ref{corvolume}), il suffit de montrer que le champ associ\'e \`a l'\'equation $P_J(q)$ pr\'eserve une $3$-forme. Nous allons nous placer dans des coordonn\'ees hamiltoniennes que nous pouvons trouver dans \cite{Okamoto}. Notons $F_{J,q}\in\mathbb{C}(x,u,v)$ la fonction rationnelle v\'erifiant $(P_J(q)):u''=F_{J,q}(x,u,u')$. Il existe $H_{J,q}\in\mathbb{C}(x,u,v)$ tel que le syst\`eme associ\'e \`a l'\'equation $P_J(q)$
\begin{align*}
\frac{du}{dx}&=v\\
\frac{dv}{dx}&=F_{J,q}
\end{align*}
 soit \'equivalent au syst\`eme hamiltonien
\begin{align*}
\frac{du}{dx}&=\frac{\partial H_{J,q}}{\partial v}\\
\frac{dv}{dx}&=-\frac{\partial H_{J,q}}{\partial u}
\end{align*}
Le tir\'e en arri\`ere du champ $\frac{\partial}{\partial x}+ v\frac{\partial}{\partial u}+F_{J,q}\frac{\partial}{\partial v}$ par l'application birationnelle de $\mathbb{C}^3$: $(x,u,v)\mapsto (x,u,\frac{\partial H_{J,q}}{\partial v})$ est le champ $Y_{J,q}:=\frac{\partial}{\partial x}+ \frac{\partial H_{J,q}}{\partial v}\frac{\partial}{\partial u}-\frac{\partial H_{J,q}}{\partial u}\frac{\partial}{\partial v}$.
Donnons la liste de ces applications:
\begin{align*}
H_{{\textsc{\romannumeral 3}},q}&= \frac{1}{x}\left[ 2u^2v^2-(2axu^2+(2b+1)u-2cx)v+a(b+d)xu\right] \\
H_{{\textsc{\romannumeral 4}},q}&= 2uv^2-(u^2+2xu+2a)v+bu\\
 H_{{\textsc{\romannumeral 5}},q}&= \frac{1}{x}\left[u(u-1)^2v^2-\left(a(u-1)^2+bu(u-1)-cxu\right)v+\frac{1}{4}\left((a+b)^2-d^2\right)(u-1)\right] \\
H_{{\textsc{\romannumeral 6}},q}&=\frac{1}{x(x-1)}\left[u(u-1)(u-x)v^2
-\left(a(u-1)(u-x)+bu(u-x)+(c-1)u(u-1)\right)\phantom{\frac{1}{4}}\right.
\\&\left.+\frac{1}{4}\left((a+b+c-1)^2-d^2\right)(u-x)\right]
 \end{align*}
 o\`{u} $(a,b,c,d)\in\mathbb{C}^4$.
Ainsi $\mathcal{L}_{Y_{J,q}}(dx\wedge du\wedge dv)=0$. Ceci conclut la preuve du th\'eor\`eme.
\end{proof}
\begin{corollary}\label{coroirred}
Pour $J={\textsc{\romannumeral 1}},{\textsc{\romannumeral 2}},{\textsc{\romannumeral 3}},{\textsc{\romannumeral 4}},{\textsc{\romannumeral 5}},{\textsc{\romannumeral 6}}$ et pour $q\in S_J$ g\'en\'eral, l'\'equation $P_J(q)$ est irr\'eductible.
\end{corollary}
\begin{proof}
Par le th\'eor\`eme (\ref{corirred}), pour $J={\textsc{\romannumeral 1}},{\textsc{\romannumeral 2}},{\textsc{\romannumeral 3}},{\textsc{\romannumeral 4}},{\textsc{\romannumeral 5}},{\textsc{\romannumeral 6}}$ et pour $q\in S_J$ g\'en\'eral, $Gal(P_J(q))=Vol(P_J(q))$. Le th\'eor\`eme (\ref{croissancered}) termine la preuve.
\end{proof}

\bibliographystyle{plain}
\bibliography{biblio}

\end{document}